%% file: AB1_2018-06-28.tex
\documentclass[a4paper,12pt]{amsart}

\input{preamble}

\begin{document}

\date{\today}
\title[Product-system models for twisted $C^*$-algebras of topological $k$-graphs]{Product-system models for twisted $C^*$-algebras of topological higher-rank graphs}

\author[Armstrong]{Becky Armstrong}
\author[Brownlowe]{Nathan Brownlowe}
\address{School of Mathematics and Statistics, The University of Sydney, NSW 2006, Australia}
\email{\href{mailto:becky.armstrong@sydney.edu.au}{becky.armstrong}, \href{mailto:nathan.brownlowe@sydney.edu.au}{nathan.brownlowe@sydney.edu.au}}

\thanks{The first author was supported by a Research Training Program Stipend Scholarship. \\ The second author was supported by the Australian Research Council grant DP170101821.}

\begin{abstract}
We use product systems of $C^*$-correspondences to introduce twisted $C^*$-algebras of topological higher-rank graphs. We define the notion of a continuous $\T$-valued $2$-cocycle on a topological higher-rank graph, and present examples of such cocycles on large classes of topological higher-rank graphs. To every proper, source-free topological higher-rank graph $\Lambda$, and continuous $\T$-valued $2$-cocycle $c$ on $\Lambda$, we associate a product system $X$ of $C_0(\Lambda^0)$-correspondences built from finite paths in $\Lambda$. We define the twisted Cuntz--Krieger algebra $C^*(\Lambda,c)$ to be the Cuntz--Pimsner algebra $\OO(X)$, and we define the twisted Toeplitz algebra $\TT C^*(\Lambda,c)$ to be the Nica--Toeplitz algebra $\NT(X)$. We also associate to $\Lambda$ and $c$ a product system $Y$ of $C_0(\Lambda^\infty)$-correspondences built from infinite paths. We prove that there is an embedding of $\TT C^*(\Lambda,c)$ into $\NT(Y)$, and an isomorphism between $C^*(\Lambda,c)$ and $\OO(Y)$. 
\end{abstract}

\maketitle

\section{Introduction} \label{sec: intro}

The theory of graph $C^*$-algebras originated through the work of Enomoto and Watatani in \cite{EW1980}, and then was continued by Kumjian, Pask, Raeburn, and Renault in \cite{KPRR1997} on the $C^*$-algebras of directed graphs. These $C^*$-algebras are generalisations of the Cuntz algebras introduced in \cite{Cuntz1977} and the Cuntz--Krieger algebras introduced in \cite{CK1980}. Since their introduction, directed graph $C^*$-algebras have been a thriving area of research within the field of operator algebras, and have been generalised in many different directions, including the $C^*$-algebras of higher-rank graphs (or $k$-graphs) by Kumjian and Pask in \cite{KP2000}, the $C^*$-algebras associated to $P$-graphs by the second-named author, Sims, and Vittadello in \cite{BSV2013}, and the $C^*$-algebras associated to topological graphs by Katsura in \cite{Katsura2004TAMS}. These classes of $C^*$-algebras significantly extend the range of $C^*$-algebras that can be described using graphs; for instance, all Kirchberg algebras can be modelled using topological graph $C^*$-algebras \cite{Katsura2008JFA} or $P$-graph $C^*$-algebras \cite{BSV2013}.

The topological graph and $k$-graph $C^*$-algebra constructions were unified by Yeend in \cite{Yeend2005Thesis, Yeend2006CM, Yeend2007JOT} with the introduction of topological $k$-graphs and their $C^*$-algebras. Yeend studied a topological path space carrying an action of $\Lambda$, and used this data to build a topological groupoid $G_\Lambda$. A closed invariant subset $\partial \Lambda$ of the unit space of $G_\Lambda$ also gives rise to the boundary-path groupoid $\GG_\Lambda \coloneqq G_\Lambda|_{\partial \Lambda}$. The groupoid $C^*$-algebras $C^*(G_\Lambda)$ and $C^*(\GG_\Lambda)$ are the Toeplitz and Cuntz--Krieger algebras associated to $\Lambda$, respectively.

In \cite{KPS2015TAMS}, Kumjian, Pask, and Sims generalised $k$-graph $C^*$-algebras in a different direction with the introduction of the twisted Cuntz--Krieger algebra $C^*(\Lambda, c)$ associated to a row-finite $k$-graph $\Lambda$ with no sources and a $\T$-valued categorical $2$-cocycle $c$ on $\Lambda$. The cohomology of $k$-graphs and the notion of using cohomological data to twist a $k$-graph $C^*$-algebra began in \cite{KPS2012JFA}, and has been further investigated in \cite{KPS2013JMAA, SWW2014, KPS2016JNG}. The twisting process results in increased structural complexity and has enabled the study of many interesting examples of $C^*$-algebras using graph algebra techniques, including all noncommutative tori and Heegaard-type quantum 3-spheres (see \cite[Section~7]{KPS2012JFA}).

In this paper, we initiate the study of twisted $C^*$-algebras of topological $k$-graphs. We work with topological $k$-graphs that are proper and source-free (the analogue of row-finite and no sources for directed graphs), and continuous $\T$-valued $2$-cocycles on these graphs. For each such topological $k$-graph $\Lambda$ and cocycle on $\Lambda$, we construct two product systems of $C^*$-correspondences: a product system $X$ built from finite paths in $\Lambda$, and a product system $Y$ built from infinite paths. The innovation in our construction comes from the cocycle, which we incorporate into the definition of the products in $X$ and $Y$. Product systems built from finite and infinite paths of higher-rank graphs are natural constructions, and our work follows the earlier work by Yamashita in \cite{Yamashita2009}, and Carlsen, Larsen, Sims, and Vittadello in \cite{CLSV2011}. Indeed, the product system studied in \cite{CLSV2011} is the ``untwisted'' version of $X$, and the product system studied in \cite{Yamashita2009} is the ``untwisted'' version of $Y$. See also \cite{Brownlowe2012} for a product system built from the boundary paths in a finitely aligned $k$-graph. 

Fowler introduced product systems of $C^*$-correspondences and their $C^*$-algebras in \cite{Fowler2002}. Applying Fowler's theory to our product system $X$ enables us to define our twisted $C^*$-algebras. In particular, we define the twisted Cuntz--Krieger algebra $C^*(\Lambda,c)$ to be the Cuntz--Pimsner algebra $\OO(X)$, and we define the twisted Toeplitz algebra $\TT C^*(\Lambda,c)$ to be the Nica--Toeplitz algebra $\NT(X)$. In our main result we study the relationship between these $C^*$-algebras and the Cuntz--Pimsner and Nica--Toeplitz algebras of the product system $Y$ built from infinite paths. We prove that $\TT C^*(\Lambda,c)$ embeds into $\NT(Y)$, and that $C^*(\Lambda,c)$ is isomorphic to $\OO(Y)$. 

We begin by providing some background on $C^*$-correspondences, product systems, and topological graphs in \autoref{sec: background}. In \autoref{sec: top k-graphs and cocycles} we recall Yeend's notion of a topological $k$-graph, and we prove some basic results about a class of topological $k$-graphs which we describe as proper and source-free. In \autoref{sec: top k-graphs and cocycles} we also introduce the notion of a continuous $\T$-valued $2$-cocycle on a topological $k$-graph, and we provide several broad classes of examples, including a new class of topological higher-rank graphs built from actions of $\Z^l$ on topological $k$-graphs. In \autoref{sec: product systems} we construct the product systems $X$ and $Y$, define our $C^*$-algebras $C^*(\Lambda,c)$ and $\TT C^*(\Lambda,c)$, and state our main theorem. In \autoref{sec: NC repn of X} we build a Nica-covariant representation $\psi$ of $X$ in the Nica--Toeplitz algebra of $Y$. In \autoref{sec: the C*s} we prove our main theorem by using $\psi$ to construct an embedding of $\TT C^*(\Lambda,c)$ into $\NT(Y)$, and an isomorphism between $C^*(\Lambda,c)$ and $\OO(Y)$.

\vspace{0.5em}
\begin{center}
\emph{Acknowledgements.} The authors would like to thank Aidan Sims for many fruitful discussions.
\end{center}

\vspace{0.5em}
\section{Background} \label{sec: background}

In this section we present the necessary background on $C^*$-correspondences, product systems of $C^*$-correspondences and their associated $C^*$-algebras, and topological graphs and their associated $C^*$-correspondences.

\subsection{\texorpdfstring{$C^*$-correspondences}{C*-correspondences}, product systems, and their associated \texorpdfstring{$C^*$-algebras}{C*-algebras}} \label{subsec: corr, prod sys, and C stars}

Let $A$ be a $C^*$-algebra. A \emph{$C^*$-correspondence over $A$}, or an \emph{$A$-correspondence}, is a right Hilbert $A$-module $X$ with a left action of $A$ on $X$ implemented by a homomorphism $\phi \colon A \to \LL(X)$, where $\LL(X)$ is the algebra of adjointable operators on $X$. We frequently write $a \cdot x$ for $\phi(a)x$. We denote the $A$-valued inner product on $X$ by $\la \cdot, \cdot \ra_X$ and the induced norm on $X$ by $\lv \cdot \rv_X$. We write $\KK(X)$ for the generalised compact operators $\clspan\{ \Theta_{x,y} : x,y \in X \} \subseteq \LL(X)$, where $\Theta_{x,y}(z) \coloneqq x \cdot \la y,z \ra_X$. 

A \emph{representation} of an $A$-correspondence $X$ in a $C^*$-algebra $B$ is a pair $(\psi,\pi)$ consisting of a linear map $\psi \colon X \to B$ and a homomorphism $\pi \colon A \to B$ such that for all $x,y \in X$, $a \in A$ we have
\[
\psi(x \cdot a) = \psi(x)\pi(a),\ \psi(x)^*\psi(y) = \pi( \la x,y \ra_X ),\, \text{ and }\ \psi(a \cdot x) = \pi(a)\psi(x).
\]
We know from \cite[Proposition~1.6]{FR1999} that $(\psi,\pi)$ induces a homomorphism $\psi^{(1)} \colon \KK(X) \to B$, characterised by $\psi^{(1)}(\Theta_{x,y}) = \psi(x)\psi(y)^*$ (see also \cite[page~202]{Pimsner1997}).

For our background on product systems, we will only discuss product systems over $\N^k$; for the same definitions and results for product systems over subsemigroups of quasi-lattice ordered groups, see \cite{Fowler2002}, where this theory originated, or \cite[Section~2]{SY2010}.

Let $A$ be a $C^*$-algebra. A \emph{product system} over $\N^k$ is a semigroup $X=\sqcup_{n\in\N^k}X_n$ such that
\begin{enumerate}[label=(\roman*)]
\item each $X_n$ is an $A$-correspondence;
\item the $A$-correspondence $X_0$ is a copy of $_A A_A$;
\item for each nonzero $m,n \in \N^k$, the map $X_m\times X_n\to X_{m+n}$ given by $x\otimes y\mapsto xy$ extends to an isomorphism of $A$-correspondences $X_m\otimes_A X_n \cong X_{m+n}$; and
\item $ax = a \cdot x$ and $xa = x \cdot a$, for each $x \in X$ and $a \in X_0$.
\end{enumerate}
We denote by $\phi_{X_n}$ the homomorphism $A \to \LL(X_n)$ implementing the left action of $A$ on $X_n$, and by $\la \cdot, \cdot \ra_{X_n}$ the inner product on $X_n$.

For nonzero $m,n \in \N^k$ with $m \le n$, there is a homomorphism $\iota_m^n \colon \LL(X_m) \to \LL(X_n)$ characterised by $\iota_m^n(S)(xy) = (Sx)y$ for all $x \in X_m$, $y \in X_{n-m}$. Since $\KK(X_0)$ can be identified with $A$, we define $\iota_0^n \coloneqq \phi_{X_n}$ for all $n \in \N^k$. For each $m,n \in \N^k$ we denote by $m \vee n$ the coordinatewise maximum of $m$ and $n$. A product system $X$ is \emph{compactly aligned} if, for all $m,n \in \N^k$, and all $S \in \KK(X_m)$ and $T \in \KK(X_n)$, we have $\iota_m^{m \vee n}(S)\iota_n^{m \vee n}(T) \in \KK(X_{m\vee n})$. Recall from \cite[Proposition~5.8]{Fowler2002} that if the left action $\phi_{X_n}$ has range in $\KK(X_n)$, for all $n \in \N^k$, then $X$ is compactly aligned.

A \emph{representation} $\psi$ of a product system $X$ in a $C^*$-algebra $B$ is a map from $X$ to $B$ such that each $(\psi|_{X_n},\psi|_{X_0})$ is a representation of the $C^*$-correspondence $X_n$, and $\psi(xy) = \psi(x)\psi(y)$ for all $x,y \in X$. For each $n \in \N^k$, we write $\psi_n$ for $\psi|_{X_n}$. We denote by $\psi^{(n)}$ the homomorphism $\KK(X_n) \to B$ characterised by $\psi^{(n)}(\Theta_{x,y}) = \psi_n(x)\psi_n(y)^*$. A representation $\psi$ of a compactly aligned product system $X$ is \emph{Nica covariant} if, for all $m,n \in \N^k$, and all $S \in \KK(X_m)$, $T \in \KK(X_n)$, we have
\[
\psi^{(m)}(S) \psi^{(n)}(T) = \psi^{(m \vee n)}( \iota_m^{m \vee n}(S) \iota_n^{m \vee n}(T)).
\]
In \cite[Theorem~6.3]{Fowler2002}, Fowler introduced the \emph{Nica--Toeplitz algebra} $\NT(X)$ (originally denoted by $\TT_{\cov}(X)$), which is the $C^*$-algebra generated by a universal Nica-covariant representation $i_X \colon X \to \NT(X)$; that is, $\NT(X)$ is generated by the image of $i_X$, and if $\psi \colon X \to B$ is a Nica-covariant representation of $X$ in a $C^*$-algebra $B$, then there exists a homomorphism $\psi^\NT \colon \NT(X) \to B$ satisfying $\psi^\NT \circ i_X = \psi$. For each $n \in \N^k$, we write $i_{X,n}$ for $i_X|_{X_n}$. We know that $i_X$ is isometric because the Fock representation is isometric (see \cite{Fowler2002}).

Suppose $X$ is a product system of $A$-correspondences and each left action $\phi_{X_n}$ is injective and has range in $\KK(X_n)$. A representation $\zeta$ of $X$ is \emph{Cuntz--Pimsner covariant} if 
\[
\zeta_0(a) = \zeta^{(n)}(\phi_{X_n}(a))\, \text{ for all } a \in A,\, n \in \N^k.
\]
The \emph{Cuntz--Pimsner algebra} $\OO(X)$ is the universal $C^*$-algebra generated by a Cuntz--Pimsner representation of $X$. Under the assumptions on the left actions, we can apply \cite[Lemma~2.2]{AaHR2018} to see that $\OO(X)$ is the quotient of $\NT(X)$ by the ideal generated by $\{ i_{X,0}(a) - i_X^{(n)}(\phi_{X_n}(a)) : n \in \N^k, a \in A \}$. We denote by $q_X$ the quotient map $\NT(X) \to \OO(X)$, and we write $j_X \coloneqq q_X \circ i_X$ for the Cuntz--Pimsner-covariant representation generating $\OO(X)$. If $\zeta$ is a Cuntz--Pimsner-covariant representation of $X$, then we denote by $\zeta^\OO$ the homomorphism of $\OO(X)$ induced from the universal property; that is, $\zeta^\OO$ satisfies $\zeta^\OO \circ j_X = \zeta$.

We note here that under the assumptions on the left actions, the Cuntz--Pimsner algebra $\OO(X)$ coincides with the Cuntz--Nica--Pimsner algebra $\NO(X)$ from \cite{SY2010}.

We finish this subsection by stating the following uniqueness theorem for representations of the Nica--Toeplitz algebra of a product system. We will use this theorem in the proof of our main theorem. This is an abridged version of \cite[Theorem~3.2~(i)]{Fletcher2017}, and is a reformulation of \cite[Theorem~7.2]{Fowler2002} in the more general setting of nonessential $C^*$-correspondences. (See also \cite{KL2017i, KL2017ii} for uniqueness theorems in a more general setting.)

\begin{theorem} \label{thm: Fletcher uniqueness theorem}
Let $X$ be a compactly aligned product system over $\N^k$ with coefficient algebra $A$. Suppose that $\psi\colon X \to B(\HH)$ is a Nica-covariant representation of $X$ on a Hilbert space $\HH$. For each $n\in \N^k$, denote by $P_n^\psi$ the projection onto $\overline{\psi_n(X_n)\HH}$. If, for any finite set $K\subseteq \N^k {\setminus} \{0\}$, the representation $\varphi_{\psi,K}$ of $A$ on $\HH$ given by
\[
\varphi_{\psi,K}\colon a\mapsto \psi_0(a)\prod_{n\in K} \big( 1_\HH - P_n^\psi)
\]
is faithful, then the induced homomorphism $\psi^\NT\colon \NT(X)\to B(\HH)$ is faithful.
\end{theorem}

\subsection{\texorpdfstring{$C^*$-correspondences}{C*-correspondences} from topological graphs} \label{subsec: corrs from top graphs}

Recall from Katsura's \cite{Katsura2004TAMS} that a \emph{topological graph} $E=(E^0,E^1,r,s)$ consists of locally compact Hausdorff spaces $E^0$ and $E^1$, a continuous map $r\colon E^1\to E^0$, and a local homeomorphism $s\colon E^1\to E^0$. In \cite{Katsura2004TAMS}, Katsura associates to each topological graph $E = (E^0, E^1, r, s)$ a $C_0(E^0)$-correspondence $X(E)$, which is the completion of $C_c(E^1)$ under the norm induced from the inner product $\la \cdot, \cdot \ra_{X(E)} \colon E^0 \to \C$ given by
\[
\la f,g \ra_{X(E)}(v) \coloneqq \sum_{e \in s^{-1}(v)} \overline{f(e)}g(e),\, \text{ for all }\, f,g \in C_c(E^1),\, v \in E^0.
\] 
On $C_c(E^1)$, the left and right actions of $C_0(E^0)$ are given by $(g \cdot f)(e) \coloneqq g(r(e))f(e)$ and $(f \cdot g)(e) \coloneqq f(e)g(s(e))$, respectively, for all $f \in C_c(E^1)$, $g \in C_0(E^0)$, $e \in E^1$. Katsura originally defines $X(E)$ to be the set of functions $\{ f \in C(E^1) : \la f, f \ra_{X(E)} \in C_0(E^0)\}$. We freely use both descriptions of $X(E)$ throughout this paper, and we call it the \emph{topological graph correspondence of $E$}.

\begin{definition} \label{def: s-section}
Let $E  = (E^0,E^1,r,s)$ be a topological graph. A subset $W$ of $E^1$ is called an \emph{$s$-section} of $E^1$ if there exists an open subset $U$ of $E^1$ containing $W$ such that the map $s|_U\colon U \to s(U)$ is a homeomorphism.
\end{definition}

\begin{notation} \label{notation: F_E}
For a topological graph $E=(E^0,E^1,r,s)$, we denote by $F_E$ the set of functions $f\in C_c(E^1)$ such that $\supp(f)$ is an $s$-section.
\end{notation}

\begin{lemma} \label{lem: nice functions for top graph mods}
Let $E=(E^0,E^1,r,s)$ be a topological graph, and $X(E)$ be the associated $C_0(E^0)$-correspondence. Then $\vecspan F_E$ is dense in $X(E)$ with respect to the module norm $\lv \cdot \rv_{X(E)}$.
\end{lemma}

\begin{proof}
A standard argument using the Stone--Weierstrass theorem shows that for any open subset $U$ of $E_1$, the subalgebra $C_c(U) \cap \vecspan(F_E)$ is uniformly dense in $C_c(U)$. Applying \cite[Lemma~1.26]{Katsura2004TAMS} now gives the result.
\end{proof}

It is straightforward to prove the following corollary of \autoref{lem: nice functions for top graph mods}.

\begin{cor} \label{cor: new compacts}
Let $E=(E^0,E^1,r,s)$ be a topological graph, and $X(E)$ be the associated $C_0(E^0)$-correspondence. Then $\KK(X(E))=\clspan\{\Theta_{f,g} : f,g\in F_E\}$.
\end{cor}

\vspace{0.5em}
\section{Topological higher-rank graphs and continuous cocycles} \label{sec: top k-graphs and cocycles}

In this section we provide some background on topological $k$-graphs, and prove a number of basic results about proper, source-free topological $k$-graphs. We then describe the construction of two different collections of $C^*$-correspondences from topological $k$-graphs. We conclude the section by defining the notion of a continuous $\T$-valued $2$-cocycle on a topological $k$-graph, and discussing several classes of examples. 

\subsection{Proper and source-free topological higher-rank graphs} \label{subsec: top graphs and top k-graphs}

Yeend unified the notions of topological graphs and $k$-graphs with the following definition, which is taken from \cite[Section~2]{Yeend2006CM}.

\begin{definition} \label{def: top k-graph}
Let $k \in \N {\setminus} \{0\}$. A \emph{topological higher-rank graph}, or \emph{topological $k$-graph}, is a pair $(\Lambda, d)$ consisting of a small category $\Lambda = (\Obj(\Lambda), \Mor(\Lambda), r, s, \circ)$ and a functor $d \colon \Lambda \to \N^k$, called the \emph{degree map}, which satisfy:
\begin{enumerate}[label=(\roman*)]
\item $\Obj(\Lambda)$ and $\Mor(\Lambda)$ are both second-countable, locally compact Hausdorff spaces;
\item $r, s \colon \Mor(\Lambda) \to \Obj(\Lambda)$ are continuous, and $s$ is a local homeomorphism;
\item the composition map $\circ \colon \Lambda \times_c \Lambda \coloneqq \{(\lambda, \mu) \in \Lambda \times \Lambda : s(\lambda) = r(\mu) \} \to \Lambda$ is continuous and open, where $\Lambda \times_c \Lambda$ has the subspace topology inherited from the product topology on $\Lambda \times \Lambda$;
\item $d$ is continuous, where $\N^k$ has the discrete topology; and
\item the \emph{unique factorisation property}: for all $\lambda \in \Lambda$ and $m, n \in \N^k$ such that $d(\lambda) = m + n$, there exists a unique pair $(\mu, \nu) \in \Lambda \times_c \Lambda$ such that $\lambda = \mu \nu$, $d(\mu) = m$, and $d(\nu) = n$.
\end{enumerate}
As is customary in the theory of $k$-graphs, we refer to the morphisms of $\Lambda$ as \emph{paths} and the objects of $\Lambda$ as \emph{vertices}. We call the domain and codomain maps of $\Lambda$ the \emph{source} and \emph{range} maps, respectively. We refer to $\Lambda \times_c \Lambda$ as the set of \emph{composable paths} in $\Lambda$. We usually simply denote the topological $k$-graph $(\Lambda,d)$ by $\Lambda$. We extend \autoref{def: s-section} to the topological $k$-graph setting by defining an \emph{$s$-section} of $\Lambda$ to be any subset $W$ of an open subset $U$ of $\Lambda$ such that the map $s|_U\colon U \to s(U)$ is a homeomorphism.
\end{definition}

\begin{notation} \label{notation: top k-graphs notation}
Let $\Lambda$ be a topological $k$-graph.
\begin{enumerate}[label=(\roman*)]
\item For each $n \in \N^k$, we define $\Lambda^n \coloneqq d^{-1}(n)$. Since $\N^k$ has the discrete topology and $d$ is continuous, $\Lambda^n$ is clopen for each $n \in \N^k$.
\item For any subset $U$ of $\Lambda$ and vertex $v \in \Lambda^0$, we define $vU \coloneqq r|_U^{-1}(v)$ and $Uv \coloneqq s|_U^{-1}(v)$.
\item For any two subsets $U$ and $V$ of $\Lambda$, we write $UV \coloneqq \{ \mu\nu : \mu \in U,\, \nu \in V,\, s(\mu) = r(\nu) \}$.
\item For $m, n \in \N^k$, $U \subseteq \Lambda^m$, and $V \subseteq \Lambda^n$, we define
\[
U \vee V \coloneqq U \Lambda^{(m \vee n) - m} \cap V \Lambda^{(m \vee n) - n} \subseteq \Lambda^{m \vee n}.
\]
\item For each $\lambda \in \Lambda$ and $m, n \in \N^k$ with $m \le n \le d(\lambda)$, the unique factorisation property implies that there are unique paths $\alpha \in \Lambda^m$, $\beta \in \Lambda^{n-m}$, and $\gamma \in \Lambda^{d(\lambda)-n}$ such that $\lambda = \alpha\beta\gamma$. We write $\lambda(0,m)$ for $\alpha$, $\lambda(m,n)$ for $\beta$, and $\lambda(n,d(\lambda))$ for $\gamma$.
\end{enumerate}
\end{notation}

\begin{definition}
A degree-preserving functor between topological $k$-graphs is called a \emph{$k$-graph morphism}. More precisely, if $\Lambda_1$ and $\Lambda_2$ are topological $k$-graphs, then a functor $f\colon \Lambda_1 \to \Lambda_2$ is a $k$-graph morphism if, for all $(\lambda,\mu) \in \Lambda_1 \times_c \Lambda_1$, we have $f(r_1(\lambda)) = r_2(f(\lambda))$, $f(s_1(\lambda)) = s_2(f(\lambda))$, $(f(\lambda),f(\mu)) \in \Lambda_2 \times_c \Lambda_2$, $f(\lambda\mu) = f(\lambda)f(\mu)$, and $d_1(\lambda) = d_2(f(\lambda))$. An \emph{automorphism} of a topological $k$-graph $\Lambda$ is a $k$-graph morphism $g\colon \Lambda \to \Lambda$ that is open, continuous, and bijective.
\end{definition}

The following definition is taken from \cite[Section~6]{Yeend2007JOT}.

\begin{definition} \label{def: proper and source-free}
Let $\Lambda$ be a topological $k$-graph. A vertex $v \in \Lambda^0$ is a \emph{source} if there is $i \in \{1,\dotsc,k\}$ such that $v\Lambda^{e_i} = \emptyset$. We say that $\Lambda$ is \emph{source-free} if none of its vertices are sources. We say that $\Lambda$ is \emph{proper} if, for each $m \in \N^k$, the range map $r|_{\Lambda^m}$ is a proper map, in the sense that the preimage of any compact subset of $\Lambda^0$ is a compact subset of $\Lambda^m$.
\end{definition}

\begin{remarks} \label{rems: about proper and s-f} \leavevmode
\begin{enumerate}[label=(\roman*)]
\item It is straightforward to show that if $\Lambda$ is a source-free topological $k$-graph, then for all $v \in \Lambda^0$ and $m \in \N^k$, we have $v\Lambda^m \ne \emptyset$.

\item Properness for topological $k$-graphs is the analogue of row-finiteness for discrete $k$-graphs. The notion of row-finiteness for topological $k$-graphs defined in \cite[Definition~2.2]{Yamashita2009} is equivalent to properness. Another equivalent characterisation of a topological $k$-graph being proper is that for each $m \in \N^k$ and $v \in \Lambda^0$, $r|_{\Lambda^m}$ is a closed map and $v\Lambda^m$ is a compact subset of $\Lambda^m$.
\end{enumerate}
\end{remarks}

The following definition is taken from \cite[Definition~2.3]{Yeend2007JOT}.

\begin{definition}
Let $\Lambda$ be a topological $k$-graph. We say that $\Lambda$ is \emph{compactly aligned} if for all $m, n \in \N^k$, and compact subsets $U \subseteq \Lambda^m$ and $V \subseteq \Lambda^n$, the set $U \vee V$ is compact in $\Lambda^{m \vee n}$.
\end{definition}

\begin{remark}
Many of the results in \cite{Yeend2006CM, Yeend2007JOT} pertain to topological $k$-graphs that are compactly aligned, and we know from \cite[Remark~6.5]{Yeend2007JOT} that any proper topological $k$-graph is compactly aligned.
\end{remark}

In \cite{Yeend2006CM, Yeend2007JOT}, Yeend constructs two groupoids from a topological $k$-graph $\Lambda$: the path groupoid $G_\Lambda$, and the boundary-path groupoid $\GG_\Lambda$. The unit space of the path groupoid is the path space $X_\Lambda$, which is defined in \cite[Definition~3.1]{Yeend2007JOT}. The unit space of the boundary-path groupoid is the boundary-path space $\partial\Lambda$, which is defined in \cite[Definition~4.2]{Yeend2007JOT}, and is a closed invariant subspace of $X_\Lambda$ (by \cite[Proposition~4.4~and~Proposition~4.7]{Yeend2007JOT}). It follows from \cite[Lemma~6.6]{Yeend2007JOT} that when $\Lambda$ is proper and source-free, the boundary-path space has a simpler characterisation, which is analogous to the definition of the infinite-path space of a $k$-graph in \cite[Definitions~2.1]{KP2000}. In this paper, we only work with topological $k$-graphs that are proper and source-free, and hence it will suffice for our purposes to only provide a definition of the infinite-path space.

As in \cite[Examples~1.7~(ii)]{KP2000}, for $k \in \N {\setminus} \{0\}$, let $\Omega_k$ be the small category with objects $\Obj(\Omega_k) \coloneqq \N^k$; morphisms $\Mor(\Omega_k) \coloneqq \{ (m,n) \in \N^k \times \N^k : m \le n \}$; range and source maps given by $r(m,n) \coloneqq m$ and $s(m,n) \coloneqq n$, respectively; and composition of morphisms given by $\circ((m,n),(n,p)) \coloneqq (m,p)$. Let $d\colon \Omega_k \to \N^k$ be given by $d(m,n) \coloneqq n-m$. Then the pair $(\Omega_k,d)$ is a $k$-graph.

\begin{definition}
Let $\Lambda$ be a proper, source-free topological $k$-graph. The \emph{infinite-path space} of $\Lambda$ is the set
\[
\Lambda^\infty \coloneqq \{ x\colon \Omega_k \to \Lambda\, :\, x \text{ is a $k$-graph morphism} \}.
\]
\end{definition}

\begin{notation}
Let $\Lambda$ be a proper, source-free topological $k$-graph. We extend the range map $r$ to $\Lambda^\infty$ via $r(x) \coloneqq x(0)$. For each $v \in \Lambda^\infty$, we define $v \Lambda^\infty \coloneqq \{ x \in \Lambda^\infty : r(x) = v \}$. Given any subset $U$ of a proper, source-free topological $k$-graph $\Lambda$, we define the \emph{cylinder set}
\[
Z(U) \coloneqq \{ x \in \Lambda^\infty : x(0,n) \in U \text{ for some } n \in \N^k \}.
\]
\end{notation}

\begin{remark} \label{rem: infinite paths coming into every vertex}
As in \cite[Section~2]{KP2000}, it follows from the fact that $\Lambda$ is source-free that for each $v \in \Lambda^0$, we have $v\Lambda^\infty \ne \emptyset$. It also follows that for each $\lambda \in \Lambda$, we have $Z(\{\lambda\}) \ne \emptyset$. Hence, for each $n \in \N^k$, we have $\Lambda^n = \{ x(0,n) : x \in \Lambda^\infty \}$.
\end{remark}

We now wish to give the infinite-path space a locally compact Hausdorff topology.

\begin{prop} \label{prop: Lambda^infty is a LCH space}
Let $\Lambda$ be a proper, source-free topological $k$-graph. The collection
\[
\{ Z(U) : \text{$U$ is an open subset of $\Lambda^n$ for some $n \in \N^k$} \}
\]
is a basis for a locally compact Hausdorff topology on $\Lambda^\infty$.
\end{prop}

\begin{proof}
Since $\Lambda^\infty$ is the boundary-path space of $\Lambda$, we can apply \cite[Proposition~3.8]{Yeend2007JOT} to $\Lambda^\infty$, and it follows that $\{ Z(U) \cap Z(F)^c : \text{$U \subseteq \Lambda$ precompact open, and $F \subseteq \overline{U}\Lambda$ compact} \}$ is a basis for a locally compact Hausdorff topology on $\Lambda^\infty$. We claim that this basis generates the same topology as $\{ Z(V) : V \subseteq \Lambda \text{ open}\}$. To see this, let $V$ be an open subset of $\Lambda$ and fix $x \in Z(V)$. Let $n \in \N^k$ with $x(0,n) \in V$. Choose a precompact open subset $U$ of $V$ such that $x(0,n) \in U \subseteq \overline{U} \subseteq V$, and take $F = \emptyset$. Then we have $x \in Z(U) \cap Z(F)^c \subseteq Z(V)$. For the other direction, fix a precompact open subset $U$ of $\Lambda$ and a compact subset $F$ of $\overline{U}\Lambda$, and let $x \in Z(U) \cap Z(F)^c$. Let $n \in \N^k$ with $x(0,n) \in U$. Since $F$ is compact and the degree map is continuous, $d(F)$ is a finite subset of $\N^k$. Let $p \in \N^k$ be the coordinatewise maximum of all of the elements of $d(F)\cup\{n\}$, and define
\[
V \coloneqq (U \cap \Lambda^n)\Lambda^{p-n} \cap \Big( \bigcap_{m \in d(F)} (\Lambda^m \setminus F)\Lambda^{p-m}\Big).
\]
Then $V$ is open, and we have $x \in Z(V) \subseteq Z(U) \cap Z(F)^c$. So the claim holds. Now, for each $n \in \N^k$ and open subset $U$ of $\Lambda$, $Z(U \cap \Lambda^n)$ is an open subset of $\Lambda^\infty$, and so the result follows from the fact that $Z(U) = \cup_{n \in \N^k} Z(U \cap \Lambda^n)$.
\end{proof}

\begin{lemma} \label{lem: U compact implies Z(U) compact}
Let $\Lambda$ be a proper, source-free topological $k$-graph. Fix $n \in \N^k$, and let $U$ be a subset of $\Lambda^n$. If $U$ is compact, then $Z(U)$ is compact, and if $U$ is precompact, then $Z(U)$ is precompact.
\end{lemma}

\begin{proof}
Since $\Lambda^\infty$ is the boundary-path space of $\Lambda$, we can apply \cite[Proposition~3.15]{Yeend2007JOT} to $\Lambda^\infty$. Since $\Lambda$ is compactly aligned, we have $Z(U)$ compact for every compact $U$. If $U$ is precompact, then $\overline{U}$ is compact, and so $Z(U)$ must be precompact because $\overline{Z(U)}$ is a closed subset of the compact set $Z(\overline{U})$.
\end{proof}

\begin{lemma} \label{lem: nice nbhds of infinite paths}
Let $\Lambda$ be a proper, source-free topological $k$-graph. For each $x \in \Lambda^\infty$ and $n \in \N^k$, there exists an open subset $U$ of $\Lambda^n$ such that $\overline{U}$ is a compact $s$-section and $Z(U)$ is an open subset of $x$.
\end{lemma}

\begin{proof}
Fix $x \in \Lambda^\infty$ and $n \in \N^k$. Let $V \subseteq \Lambda^n$ be an open $s$-section containing $x(0,n)$. Since $\Lambda$ is a locally compact Hausdorff space, there is an open subset $U$ of $V$ such that $U$ is precompact, and $x(0,n) \in U \subseteq \overline{U} \subseteq {V}$. Hence $\overline{U}$ is a compact $s$-section, and $x \in Z(U)$.
\end{proof}

We now introduce several maps that will regularly be used throughout this paper.

\begin{lemma} \label{lem: tau maps}
Let $\Lambda$ be a proper, source-free topological $k$-graph. Fix $m, n \in \N^k$ with $m \le n$. The map $\tau_{m,n}\colon \Lambda^n \to \Lambda^{n-m}$ defined by $\tau_{m,n}(\lambda) \coloneqq \lambda(m,n)$ is a local homeomorphism.
\end{lemma}

\begin{proof}
We will begin by showing that $\tau_{m,n}$ is continuous. Let $U$ be an open subset of $\Lambda^{n-m}$. Then we have $\tau_{m,n}^{-1}(U) = \Lambda^m U$, which is an open subset of $\Lambda^n$ because composition is an open map. Hence $\tau_{m,n}$ is continuous.

Fix $\lambda \in \Lambda^n$. Let $V$ be an open $s$-section such that $\lambda \in V \subseteq \Lambda^n$. Since $\Lambda$ is a locally compact Hausdorff space, there exists an open subset $W$ of $V$ such that $\lambda \in W \subseteq \overline{W} \subseteq V$. We claim that $\tau_{m,n}|_{\overline{W}}$ is a homeomorphism. To see that $\tau_{m,n}|_{\overline{W}}$ is injective, suppose that for $\mu, \nu \in \overline{W}$, we have $\tau_{m,n}|_{\overline{W}}(\mu) = \tau_{m,n}|_{\overline{W}}(\nu)$. Then $\mu(m,n) = \nu(m,n)$, and so $s|_{\overline{W}}(\mu) = s|_{\overline{W}}(\nu)$. Hence we have $\mu = \nu$, because $\overline{W}$ is an $s$-section. Thus $\tau_{m,n}|_{\overline{W}}$ is injective. Since $\tau_{m,n}|_{\overline{W}}\colon \overline{W} \to \tau_{m,n}(\overline{W})$ is a continuous bijective map from a compact space to a Hausdorff space, it is a homeomorphism. Therefore, $\tau_{m,n}$ is a local homeomorphism, because it restricts to a homeomorphism on the open set $W$.
\end{proof}

\begin{lemma} \label{lem: rho maps}
Let $\Lambda$ be a proper, source-free topological $k$-graph. Fix $m, n \in \N^k$ with $m \le n$. The maps $\rho_{m,n}\colon \Lambda^n \to \Lambda^m$ and $\rho_{m,\infty}\colon \Lambda^\infty \to \Lambda^m$, given by $\rho_{m,n}(\lambda) \coloneqq \lambda(0,m)$ and $\rho_{m,\infty}(x) \coloneqq x(0,m)$, respectively, are both continuous and proper.
\end{lemma}

\begin{proof}
Let $U$ be a subset of $\Lambda^m$. We first show that $\rho_{m,n}$ is continuous and proper. We have $\rho_{m,n}^{-1}(U) = U\Lambda^{n-m} = U r|_{\Lambda^{n-m}}^{-1}(s(U))$. If $U$ is open, then since $r|_{\Lambda^{n-m}}$ is continuous, and composition and $s$ are open maps, $\rho_{m,n}^{-1}(U)$ is an open subset of $\Lambda^n$. Hence $\rho_{m,n}$ is continuous. If $U$ is compact, then since $r|_{\Lambda^{n-m}}$ is a proper map, and composition and $s$ are continuous, $\rho_{m,n}^{-1}(U)$ is a compact subset of $\Lambda^n$. Hence $\rho_{m,n}$ is proper. We now show that $\rho_{m,\infty}$ is continuous and proper. We have $\rho_{m,\infty}^{-1}(U) = Z(U)$. If $U$ is open, then $Z(U)$ is open, and hence $\rho_{m,\infty}$ is continuous. If $U$ is compact, then \autoref{lem: U compact implies Z(U) compact} implies that $Z(U)$ is compact, and hence $\rho_{m,\infty}$ is proper.
\end{proof}

The following proposition introduces \emph{shift maps} on the infinite-path space of a proper, source-free topological $k$-graph. Shift maps on the path space of a topological $k$-graph were introduced in \cite[Lemma~3.3]{Yeend2007JOT}, and it is clear from this definition that if the domain of each shift map is restricted to the infinite-path space, then the ranges of these maps will also be in the infinite-path space.

\begin{prop} \label{prop: shift maps}
Let $\Lambda$ be a proper, source-free topological $k$-graph. For each $p \in \N^k$, there is a local homeomorphism $T^p\colon \Lambda^\infty \to \Lambda^\infty$ such that $T^p(x)(m,n) = x(m+p,n+p)$, for all $(m,n) \in \Mor(\Omega_k)$. If $p \in \N^k$ and $U$ is an $s$-section contained in $\Lambda^p$, then $T^p|_{Z(U)}$ is injective.
\end{prop}

\begin{proof}
Fix $p \in \N^k$, and $x \in \Lambda^\infty$. A straightforward argument shows that $T^p(x)\colon \Omega_k \to \Lambda$ is a $k$-graph morphism, and so $T^p(x) \in \Lambda^\infty$.

We need to show that $T^p$ is a local homeomorphism. By \autoref{lem: nice nbhds of infinite paths}, we can choose an open subset $U$ of $\Lambda^p$ such that $\overline{U}$ is a compact $s$-section and $x \in Z(U)$. We claim that $T^p|_{Z(U)}$ is a homeomorphism onto $T^p(Z(U))$. The injectivity of $T^p|_{Z(U)}$ follows from the injectivity of $s|_U$. For continuity, let $\VV$ be an open subset of $T^p(Z(U)) = Z(s(U))$. Write $\VV = \bigcup_{\alpha \in I} Z(V_\alpha)$ with each $V_\alpha$ a nonempty open subset of $\Lambda^{m_\alpha}$, for some $m_\alpha \in \N^k$. A straightforward calculation shows that 
\[
\left(T^p|_{Z(U)}\right)^{-1}\!(\VV) = \bigcup_{\alpha \in I} \left(Z(\Lambda^p\, V_\alpha) \cap Z(U)\right).
\]
Since composition is an open map, $\Lambda^p\, V_\alpha$ is an open subset of $\Lambda^{p+m_\alpha}$. Thus, for each $\alpha \in I$, $Z(\Lambda^p\, V_\alpha) \cap Z(U)$ is an open subset of $Z(U)$. Hence $\left(T^p|_{Z(U)}\right)^{-1}\!(\VV)$ is open in $Z(U)$, and so $T^p|_{Z(U)}$ is continuous.

To see that $T^p|_{Z(U)}$ is an open map, let $\WW$ be an open subset of $Z(U)$. Write $\WW= \bigcup_{\beta \in J} Z(W_\beta)$ with each $W_\beta$ a nonempty open subset of $\Lambda^{n_\beta}$, for some $n_\beta \in \N^k$. We use the maps $\tau_{p,(p \vee n_\beta)}$ from \autoref{lem: tau maps} to write
\begin{align*}
T^p|_{Z(U)}(\WW) &= \bigcup_{\beta \in J} T^p|_{Z(U)} \left(Z(W_\beta)\right) \\
&= \bigcup_{\beta \in J} T^p|_{Z(U)} \left(Z\!\left((W_\beta\, \Lambda^{(p \vee n_\beta) - n_\beta})\, \cap\, (U \Lambda^{(p \vee n_\beta) - p})\right)\right) \\
&= \bigcup_{\beta \in J} Z\!\left(\tau_{p, (p \vee n_\beta)}\!\left((W_\beta\, \Lambda^{(p \vee n_\beta) - n_\beta})\, \cap\, (U \Lambda^{(p \vee n_\beta) - p})\right)\right)\!.
\end{align*}
Since composition and the $\tau_{p,(p \vee n_\beta)}$ maps are open, it follows that $T^p|_{Z(U)}(\WW)$ is an open subset of $T^p(Z(U))$, and hence $T^p|_{Z(U)}$ is an open map. Therefore, $T^p$ is a local homeomorphism.
\end{proof}

\begin{remark}
Let $\Lambda$ be a proper, source-free topological $k$-graph. The same argument used to prove \cite[Proposition~2.3]{KP2000} shows that for all $x \in \Lambda^\infty$ and $\lambda \in \Lambda r(x)$, there is a unique $y \in \Lambda^\infty$ such that $x = T^{d(\lambda)}(y)$ and $\lambda = y(0,d(\lambda))$, and we write $y = \lambda x$.
\end{remark}

\subsection{\texorpdfstring{$C^*$-correspondences}{C*-correspondences} from topological higher-rank graphs} \label{subsec: corrs from top k-graphs}

To each proper, source-free topological $k$-graph we associate two families of topological graphs; one using finite paths in the graph, and the other using infinite paths.

\begin{lemma} \label{lem: top graphs from top k-graph}
Let $\Lambda$ be a proper, source-free topological $k$-graph, and $n\in \N^k$. Then $\Lambda_n \coloneqq (\Lambda^0,\Lambda^n,r|_{\Lambda^n},s|_{\Lambda^n})$ and $\Lambda_{\infty,n} \coloneqq (\Lambda^\infty,\Lambda^\infty,T^0,T^n)$ are topological graphs, with associated topological graph correspondences $X_n\coloneqq X(\Lambda_n)$ and $Y_n\coloneqq X(\Lambda_{\infty,n})$. The homomorphisms implementing the left actions, $\phi_{X_n}\colon C_0(\Lambda^0)\to \LL(X_n)$ and $\phi_{Y_n}\colon C_0(\Lambda^\infty)\to \LL(Y_n)$, are both injective and have range in the compact operators.
\end{lemma}

\begin{proof}
Since $d$ is continuous, $\Lambda^0$ and $\Lambda^n$ are open subsets of $\Lambda$, and hence are locally compact and Hausdorff. Standard arguments show that $r|_{\Lambda^n}$ and $s|_{\Lambda^n}$ are continuous, and that $s|_{\Lambda^n}$ is a local homeomorphism. Thus each $\Lambda_n$ is a topological graph. 

We know from \autoref{prop: Lambda^infty is a LCH space} that $\Lambda^\infty$ is a locally compact Hausdorff space. The map $T^0$ is just the identity map, and hence is continuous. We know from \autoref{prop: shift maps} that each $T^n$ is a local homeomorphism. Thus each $\Lambda_{\infty, n}$ is a topological graph.

For the claims about the left actions, we use \cite[Proposition~1.24]{Katsura2004TAMS}. Since $\Lambda$ is source-free, we have $(\Lambda^0)_{\sce}=\emptyset$, and hence $\phi_{X_n}$ is injective. Since $\Lambda$ is proper, we have $(\Lambda^0)_{\fin}=\Lambda^0$, and hence the image of each $\phi_{X_n}$ is $\KK(X_n)$. Since the range map of each $\Lambda_{\infty,n}$ is the identity, we obviously have $(\Lambda^\infty)_{\sce}=\emptyset$ and $(\Lambda^\infty)_{\fin}=\Lambda^\infty$, and hence each $\phi_{Y_n}$ is injective and has range equal to $\KK(Y_n)$. 
\end{proof}

\subsection{Cocycles and examples} \label{subsec: cocycles and examples}

The following definition is inspired by the cohomology theory for $k$-graphs developed in \cite{KPS2015TAMS, SWW2014}. 

\begin{definition} \label{def: twist}
Fix $k \in \N {\setminus} \{0\}$ and let $\Lambda$ be a topological $k$-graph. A \emph{continuous $\T$-valued $2$-cocycle} on $\Lambda$ is a continuous map $c\colon \Lambda \times_c \Lambda \to \T$ satisfying
\begin{enumerate}[label=(C\arabic*)]
\item \label{item: twist_identity} $c(\lambda, \mu) c(\lambda \mu, \nu) = c(\lambda, \mu \nu) c(\mu, \nu)$, whenever $s(\lambda) = r(\mu)$ and $s(\mu) = r(\nu)$; and
\item \label{item: twist_one} $c(\lambda, s(\lambda)) = c(r(\lambda), \lambda) = 1$, for all $\lambda \in \Lambda$.
\end{enumerate}
We call \autoref{item: twist_identity} the \emph{$2$-cocycle identity}, and we say that $c$ is \emph{normalised} because it satisfies \autoref{item: twist_one}.
\end{definition}

\begin{example} \label{eg: FPS graphs}
We begin our discussion on examples with a new class of topological higher-rank graphs, which are a topological analogue of the $(k+l)$-graphs from \cite{FPS2009} coming from actions of $\Z^l$ on $k$-graphs. Fix $k, l \in \N {\setminus} \{0\}$. Let $\Lambda$ be a topological $k$-graph, and $\beta$ an action of $\Z^l$ by automorphisms of $\Lambda$. For $p \in \N^k$ and $m \in \N^l$, we write $(p,m)$ for $(p_1,\dotsc,p_k,m_1,\dotsc,m_l) \in \N^{k+l}$.

We wish to define a topological $(k+l)$-graph $\Gamma \coloneqq \Lambda \times_\beta \Z^l$. We define $\Obj(\Gamma) \coloneqq \Lambda^0 \times \{0\}$, and $\Mor(\Gamma) \coloneqq \Lambda \times \N^l$, giving both the product topology. The range and source maps are given by $r(\mu,m) \coloneqq (r_\Lambda(\mu),0)$ and $s(\mu,m) \coloneqq (\beta_{-m}(s_\Lambda(\mu)),0)$, respectively; composition is given by $(\mu,m)(\nu,n) \coloneqq (\mu\beta_m(\nu),m+n)$, for all $\mu, \nu \in \Lambda$ such that $s_\Lambda(\mu)=r_\Lambda(\beta_m(\nu))$; and the degree map is given by $d(\mu,m) \coloneqq (d_\Lambda(\mu),m)$. We claim that $\Gamma$ is a topological $(k+l)$-graph.

The set of objects and morphisms of $\Gamma$ are second-countable, locally compact Hausdorff spaces, being the product of such spaces. It is straightforward to show that the range map is continuous and the source map is open and continuous. The source map is locally injective because for each $(\mu,m) \in \Gamma$, $s$ is injective on $V \times \{m\}$, where $V$ is an open $s_\Lambda$-section in $\Lambda$ containing $\mu$. Hence $s$ is a local homeomorphism.

To see that composition is open, let $U$ and $V$ be open subsets of $\Lambda$, and $m,n \in \N^l$. Then 
\[
\circ\Big((\Gamma \times_c \Gamma)\, \cap\, \big((U \times \{m\}) \times (V \times \{n\} )\big)\Big) = \circ_\Lambda \Big((\Lambda \times_c \Lambda)\, \cap\, (U \times \beta_m(V)) \Big) \times \{m+n\},
\]
which is open, because $\circ_\Lambda$ is an open map. To see that composition is continuous, let $U$ be an open subset of $\Lambda$, and $p \in \N^l$. Fix $m \in \N^l$ with $m \le p$. We know that $\id \times \beta_{-m}\colon \Lambda \times_c \Lambda \to \Lambda \times \Lambda$ is an open map, as is the map $g_{m,p-m}\colon \Lambda \times \Lambda \to \Gamma \times \Gamma$ given by $g_{m,p-m}(\mu,\nu) \coloneqq ((\mu,m),(\nu,p-m))$. We have
\[
\circ^{-1}(U \times \{p\}) = (\Gamma \times_c \Gamma)\, \cap\, \Big( \bigcup_{\substack{m \in \N^l, \\ m \le p}} \big(g_{m,p-m} \circ (\id \times \beta_{-m})\big)\big(\!\circ_\Lambda^{-1}(U)\big) \Big),
\]
which is open, because $\circ_\Lambda$ is continuous, and each $g_{m,p-m} \circ (\id \times \beta_{-m})$ is an open map. Hence composition is continuous. To see that the degree map is continuous, let $p \in \N^k$ and $m \in \N^l$. We have $d^{-1}(p,m) = d_\Lambda^{-1}(p) \times \{m\}$, which is open because $d_\Lambda$ is continuous. Hence $d$ is continuous. Finally, the unique factorisation property follows from the arguments on the unique factorisation property from \cite{FPS2009} (as this property does not involve the topology of $\Gamma$). Thus $\Gamma$ is a topological $(k+l)$-graph.

We claim that if $\Lambda$ is proper and source-free, then $\Gamma$ is proper and source-free. Suppose that $\Lambda$ is source-free. Let $v \in \Lambda^0$, and $i \in \{1,\dotsc,k+l\}$. If $1 \le i \le k$, then $(v,0)\Gamma^{e_i} = v\Lambda^{e_i} \times \{0\} \ne \emptyset$. If $k+1 \le i \le k+l$, then $(v,0)\Gamma^{e_i} = \{ (v,e_{i-k}) \} \ne \emptyset$. Hence $\Gamma$ is source-free. Now suppose that $\Lambda$ is proper. Let $W$ be a compact subset of $\Gamma^0$, $p \in \N^k$, and $m \in \N^l$. Then there is a compact subset $K$ of $\Lambda^0$ such that $W = K \times \{0\}$. We have 
\[
r|_{\Gamma^{(p,m)}}^{-1}(W) = r_\Lambda|_{\Lambda^p}^{-1}(K) \times \{m\},
\]
which is compact. Hence $\Gamma$ is proper.

We can construct a number of continuous $\T$-valued $2$-cocycles on $\Gamma$. For each $q \in \N{\setminus}\{0\}$ and $m \in \N^q$, we define $\lav m \rav \coloneqq \sum_{i=1}^q m_i$. Suppose that $f\colon \Lambda \to \T$ is a continuous functor such that $f \circ \beta_m = f$ for all $m \in \Z^l$. For an example of such a continuous functor, take $f$ given by $f(\mu) \coloneqq e^{i \lav d(\mu) \rav}$. We define $c_f$ on $\Gamma \times_c \Gamma$ by
\[
c_f((\mu,m),(\nu,n)) \coloneqq f(\nu)^{\lav m \rav}.
\]
We first check the cocycle identities. For \autoref{item: twist_identity}, let $\big((\mu,m),(\nu,n)\big),\,\big((\nu,n),(\lambda,p)\big) \in \Gamma \times_c \Gamma$, we have
\begin{align*}
c_f\big( (\mu,m),(\nu,n) \big) c_f\big( (\mu\beta_m(\nu),m+n), (\lambda,p) \big) &= f(\nu)^{\lav m \rav}f(\lambda)^{\lav m+n \rav} \\
&= f(\nu\beta_n(\lambda))^{\lav m \rav}f(\lambda)^{\lav n \rav} \\
&= c_f\big( (\mu,m),(\nu\beta_n(\lambda),n+p) \big) c_f\big( (\nu,n),(\lambda,p) \big),
\end{align*}
which is \autoref{item: twist_identity}. For each $(\mu,m) \in \Gamma$, we use that $f|_{\Lambda^0} \equiv 1$ to get 
\[
c_f\big( (\mu,m),(\beta_{-m}(s_{\Lambda}(\mu)),0) \big) = f(s_\Lambda(\mu))^{\lav m \rav}=1 = f(\mu)^0 = c_f\big( (r_\Lambda(\mu), 0), (\mu,m) \big).
\]
Hence \autoref{item: twist_one} is satisfied. The continuity of $c_f$ follows from the continuity of $f$, and hence $c_f$ is a continuous $\T$-valued $2$-cocycle on $\Gamma$. 

Now suppose that $\omega \colon \N^l \to \T$ is a homomorphism. We define $c_\omega$ on $\Gamma \times_c \Gamma$ by 
\[
c_\omega((\mu,m),(\nu,n)) \coloneqq \omega(m)^{\lav d(\nu) \rav}.
\]
For \autoref{item: twist_identity}, let $\big((\mu,m),(\nu,n)\big),\,\big((\nu,n),(\lambda,p)\big) \in \Gamma \times_c \Gamma$. Then 
\begin{align*}
c_\omega\big( (\mu,m),(\nu,n) \big) c_\omega\big( (\mu\beta_m(\nu),m+n), (\lambda,p) \big) &= \omega(m)^{\lav d(\nu) \rav}\,\omega(m+n)^{\lav d(\lambda) \rav} \\
&= \omega(m)^{(\lav d(\nu) \rav + \lav d(\lambda) \rav)}\,\omega(n)^{\lav d(\lambda) \rav} \\
&= \omega(m)^{\lav d(\nu\beta_n(\lambda)) \rav}\,\omega(n)^{\lav d(\lambda) \rav} \\
&= c_\omega\big( (\mu,m),(\nu\beta_n(\lambda),n+p) \big) c_\omega\big( (\nu,n),(\lambda,p) \big),
\end{align*}
and hence \autoref{item: twist_identity} is satisfied. For each $(\mu,m) \in \Gamma$, we have
\[
c_\omega\big( (\mu,m),(\beta_{-m}(s_{\Lambda}(\mu)),0) \big) = \omega(m)^{\lav d(\beta_{-m}(s_{\Lambda}(\mu))) \rav}= 1 = \omega(0)^{\lav d(\mu)\rav} = c_f\big( (r_\Lambda(\mu), 0), (\mu,m) \big).
\]
Hence \autoref{item: twist_one} is satisfied. The continuity of $c_\omega$ follows from the continuity of the degree map, and hence $c_\omega$ is a continuous $\T$-valued $2$-cocycle on $\Gamma$. 

Finally, if we let $\sigma$ be any $\T$-valued $2$-cocycle on $\Z^l$, then it is straightforward to check that $c_\sigma$ defined on $\Gamma \times_c \Gamma$ by $c_\sigma\big((\mu,m),(\nu,n)\big) \coloneqq \sigma(m,n)$ is a continuous $\T$-valued $2$-cocycle on $\Gamma$.
\end{example}

\begin{example} \label{eg: skew product}
We recall the skew-product graphs from \cite[Definition~8.1]{Yeend2007JOT}. Let $\Lambda$ be a topological $k$-graph, $A$ a locally compact group, and $f \colon \Lambda \to A$ a continuous functor. Then we can form a topological $k$-graph $\Lambda \times_f A$, which has objects $\Lambda^0 \times A$ and morphisms $\Lambda \times A$, both with the product topology; range and source maps given by $r(\mu,a) \coloneqq (r(\mu),a)$ and $s(\mu,a) \coloneqq (s(\mu),af(\mu))$; composition given by $(\mu,a)(\nu,af(\mu)) \coloneqq (\mu\nu,a)$, for all $\mu,\nu \in \Lambda$ with $s(\mu) = r(\nu)$; and degree map given by $d(\mu,a) \coloneqq d(\mu)$. We claim that if $\Lambda$ is proper and source-free, then $\Lambda \times_f A$ is proper and source-free. For each $(v,a) \in \Lambda^0 \times A$, and $i \in \{1,\dotsc,k\}$, we have $(v,a)(\Lambda \times_f A)^{e_i} = v\Lambda^{e_i} \times \{a\}$, which is nonempty because $\Lambda$ is source-free. Hence $\Lambda \times_f A$ is source-free. Now suppose $W$ is a compact subset of $\Lambda^0 \times A$, and $m \in \N^k$. Then
\[
W(\Lambda \times_f A)^m = r|_{(\Lambda \times_f A)^m}^{-1}(W) \subseteq r|_{\Lambda^m}^{-1}(\pi_1(W)) \times \pi_2(W),
\]
where $\pi_1$ and $\pi_2$ are the projections onto $\Lambda$ and $A$, respectively. Since $\Lambda$ is proper, we know that $r|_{\Lambda^m}^{-1}(\pi_1(W))$ is compact. Now, $W$ is closed because it is a compact subset of a Hausdorff space, and so the continuity of $r|_{(\Lambda \times_f A)^m}$ implies that $W(\Lambda \times_f A)^m$ is a closed subset of the compact set $r|_{\Lambda^m}^{-1}(\pi_1(W)) \times \pi_2(W)$. Hence $W(\Lambda \times_f A)^m$ is compact, and so $\Lambda \times_f A$ is proper.

For any continuous $\T$-valued $2$-cocycle $c$ on $\Lambda$, there is a continuous $\T$-valued $2$-cocycle $\tilde{c}$ on $\Lambda \times_f A$ given by $\tilde{c}((\mu,a),(\nu,b)) \coloneqq c(\mu,\nu)$, for all $((\mu,a),(\nu,b)) \in (\Lambda \times_f A) \times_c (\Lambda \times_f A)$.
\end{example}

\begin{example} \label{eg: Cartesian product graphs}
Recall from \cite[Proposition~3.2.1]{Yeend2005Thesis} that, given a topological $k_1$-graph $\Lambda_1$, and a topological $k_2$-graph $\Lambda_2$, we can form the Cartesian product $\Lambda_1 \times \Lambda_2$, which is a topological $(k_1+k_2)$-graph under the product topology, with the obvious definitions of range, source, composition, and degree maps. We claim that if $\Lambda_1$ and $\Lambda_2$ are both proper and source-free, then so is $\Lambda_1 \times \Lambda_2$. Suppose that $(u,v) \in \Lambda_1^0 \times \Lambda_2^0$, and $i \in \{1,\dotsc,k_1+k_2\}$. If $1 \le i \le k_1$, then $(u,v)(\Lambda_1 \times \Lambda_2)^{e_i} = u\Lambda_1^{e_i} \times \{v\} \ne \emptyset$. If $k_1 +1 \le i \le k_1 + k_2$, then $(u,v)(\Lambda_1 \times \Lambda_2)^{e_i} = \{u\} \times v\Lambda_2^{e_{(i-k_1)}}\ne \emptyset$. Hence $\Lambda_1 \times \Lambda_2$ is source-free. To see that $\Lambda_1 \times \Lambda_2$ is proper, let $W$ be a compact subset of $\Lambda_1^0 \times \Lambda_2^0$, and let $m=(m_1,m_2) \in \N^{k_1+k_2}$, with each $m_i \in \N^{k_i}$. Since $W$ is a compact subset of a Hausdorff space, it is closed, and so the continuity of the range map implies that $(r_1 \times r_2)|_{(\Lambda_1 \times \Lambda_2)^m}^{-1}(W)$ is closed. Hence $(r_1 \times r_2)|_{(\Lambda_1 \times \Lambda_2)^m}^{-1}(W)$ is compact, because it is a closed subset of the compact set $r_1|_{\Lambda_1^{m_1}}^{-1}(\pi_1(W)) \times r_2|_{\Lambda_2^{m_2}}^{-1}(\pi_2(W))$. Thus $\Lambda_1 \times \Lambda_2$ is proper.

We claim that if $c_1$ and $c_2$ are continuous $\T$-valued $2$-cocycles on $\Lambda_1$ and $\Lambda_2$, respectively, then $c_1 \times c_2$ given by $(c_1 \times c_2)((\lambda_1,\mu_1),(\lambda_2,\mu_2)) \coloneqq c_1(\lambda_1,\lambda_2)c_2(\mu_1,\mu_2)$ is a continuous $\T$-valued $2$-cocycle on $\Lambda_1 \times \Lambda_2$. It is straightforward to see that $c_1 \times c_2$ is continuous. For \autoref{item: twist_identity}, let $((\lambda_1,\mu_1),(\lambda_2,\mu_2)),\, ((\lambda_2,\mu_2),(\lambda_3,\mu_3)) \in (\Lambda_1 \times \Lambda_2) \times_c (\Lambda_1 \times \Lambda_2)$. We have 
\begin{align*}
(c_1 \times c_2)((\lambda_1,\mu_1),(\lambda_2,\mu_2))&\,(c_1 \times c_2)((\lambda_1\lambda_2,\mu_1\mu_2),(\lambda_3,\mu_3)) \\
&= c_1(\lambda_1,\lambda_2)c_2(\mu_1,\mu_2)c_1(\lambda_1\lambda_2,\lambda_3)c_2(\mu_1\mu_2,\mu_3) \\
&= c_1(\lambda_1,\lambda_2\lambda_3)c_1(\lambda_2,\lambda_3)c_2(\mu_1,\mu_2\mu_3)c_2(\mu_2,\mu_3) \\
&= (c_1 \times c_2)((\lambda_1,\mu_1),(\lambda_2\lambda_3,\mu_2\mu_3))\,(c_1 \times c_2)((\lambda_2,\mu_2),(\lambda_3,\mu_3)),
\end{align*}
and hence \autoref{item: twist_identity} is satisfied. For each $(\lambda,\mu) \in \Lambda_1\times \Lambda_2$, we have 
\[
(c_1 \times c_2) ((\lambda,\mu),(s_1(\lambda),s_2(\mu)))=c_1(\lambda,s_1(\lambda))c_2(\mu,s_2(\mu))=1,
\]
and
\[
(c_1 \times c_2)((r_1(\lambda),r_2(\mu)),(\lambda,\mu)) = c_1(r_1(\lambda),\lambda)c_2(r_2(\mu),\mu)=1.
\]
Hence \autoref{item: twist_one} is satisfied, and so $c_1 \times c_2$ is a continuous $\T$-valued $2$-cocycle on $\Lambda_1 \times \Lambda_2$.
\end{example}

\vspace{0.5em}
\section{Product systems and twisted \texorpdfstring{$C^*$-algebras}{C*-algebras} associated to topological higher-rank graphs} \label{sec: product systems}

In this section we define the twisted Cuntz--Krieger algebra $C^*(\Lambda,c)$ and the twisted Toeplitz algebra $\TT C^*(\Lambda,c)$ associated to a proper, source-free topological $k$-graph $\Lambda$ and a continuous $\T$-valued $2$-cocycle $c$. We also state our main theorem. We start by associating two compactly aligned product systems to $\Lambda$ and $c$.

\subsection{The product systems} \label{subsec: product systems} We now introduce the product system built from finite paths in $\Lambda$. When $c$ is trivial, \autoref{prop: finite path product system} is exactly \cite[Proposition~5.9]{CLSV2011}, but for nontrivial cocycles, we do have to work a little harder to get the result. However, the proofs of \autoref{prop: finite path product system} and \autoref{prop: infinite path product system} follow similar arguments, so we will only provide the details of the proof of \autoref{prop: infinite path product system}. 

\begin{prop} \label{prop: finite path product system}
Suppose that $\Lambda$ is a proper, source-free topological $k$-graph, and $c$ is a continuous $\T$-valued $2$-cocycle on $\Lambda$. For each $n \in \N^k$, let $X_n$ be the $C_0(\Lambda^0)$-correspondence associated to the topological graph $\Lambda_n=(\Lambda^0, \Lambda^n, r|_{\Lambda^n}, s|_{\Lambda^n})$, as in \autoref{lem: top graphs from top k-graph}. For $f \in X_m$ and $g \in X_n$, define $fg\colon \Lambda^{m+n} \to \C$ by
\[
(fg)(\lambda) \coloneqq c(\lambda(0,m),\lambda(m,m+n)) f(\lambda(0,m)) g(\lambda(m,m+n)).
\]
Then $fg \in X_{m+n}$, and under this multiplication, the family
\[
X \coloneqq \bigsqcup_{n\in\N^k} X_n
\]
of $C_0(\Lambda^0)$-correspondences is a product system over $\N^k$.
\end{prop}

We now introduce the product system built from infinite paths.

\begin{prop} \label{prop: infinite path product system}
Suppose that $\Lambda$ is a proper, source-free topological $k$-graph, and $c$ is a continuous $\T$-valued $2$-cocycle on $\Lambda$. For each $n \in \N^k$, let $Y_n$ be the $C_0(\Lambda^\infty)$-correspondence associated to the topological graph $\Lambda_{\infty,n}=(\Lambda^\infty, \Lambda^\infty, T^0, T^n)$, as in \autoref{lem: top graphs from top k-graph}. For $f \in Y_m$ and $g \in Y_n$, define $fg\colon \Lambda^\infty \to \C$ by
\[
(fg)(x) \coloneqq c(x(0,m),x(m,m+n)) f(x) g(T^m(x)).
\]
Then $fg \in Y_{m+n}$, and under this multiplication, the family 
\[
Y \coloneqq \bigsqcup_{n\in\N^k} Y_n
\]
of $C_0(\Lambda^\infty)$-correspondences is a product system over $\N^k$.
\end{prop}

\begin{proof}
Let $m,n \in \N^k$. We show that $f\otimes g \mapsto fg$ extends to an isomorphism of $Y_m \otimes_{C_0(\Lambda^\infty)} Y_n$ onto $Y_{m+n}$. Standard arguments show that for each $f \in Y_m$ and $g \in Y_n$ we have $fg$ continuous. We claim that for each $f \in Y_m$ and $g \in Y_n$, we have $\la fg, fg \ra_{Y_{m+n}} \in C_0(\Lambda^\infty)$. Let $f_1,f_2 \in Y_m$ and $g_1,g_2 \in Y_n$. Then
\begin{align*}
&\la f_1 g_1, f_2 g_2 \ra_{Y_{m+n}} (x) \\
&= \sum_{\substack{ y \in \Lambda^\infty, \\ T^{m+n}(y) = x }} \overline{c\big(y(0, m), y(m, m+n)\big) f_1(y) g_1(T^m(y))} c\big(y(0, m), y(m, m+n)\big) f_2(y) g_2(T^m(y)) \\
&= \sum_{\substack{ y \in \Lambda^\infty, \\ T^{m+n}(y) = x }} \big\lav c\big(y(0, m), y(m, m+n)\big) \big\rav^2\, \overline{f_1(y) g_1(T^m(y))} f_2(y) g_2(T^m(y)) \\
&= \sum_{\nu \in \Lambda^n r(x)} \Big( \sum_{\mu \in \Lambda^m r(\nu)} \overline{f_1(\mu\nu x)} f_2(\mu\nu x) \Big)\, \overline{g_1(\nu x)} g_2(\nu x) \\
&= \sum_{\nu \in \Lambda^n r(x)} \Big( \sum_{\substack{y \in \Lambda^\infty, \\ T^m(y) = \nu x}} \overline{f_1(y)} f_2(y) \Big)\, \overline{g_1(\nu x)} g_2(\nu x) \\
&= \sum_{\nu \in \Lambda^n r(x)} \la f_1, f_2 \ra_{Y_m} (\nu x) \overline{g_1(\nu x)} g_2(\nu x) \\
&= \sum_{\nu \in \Lambda^n r(x)} \overline{g_1(\nu x)} (\la f_1, f_2 \ra_{Y_m} \cdot g_2)(\nu x) \\
&= \sum_{\substack{y \in \Lambda^\infty, \\ T^n(y) = x}} \overline{g_1(y)} (\la f_1, f_2 \ra_{Y_m} \cdot g_2)(y) \\
&= \la g_1, \la f_1, f_2 \ra_{Y_m} \cdot g_2 \ra_{Y_n}(x).
\end{align*}
So we have
\begin{equation} \label{eq: tps_inner_product}
\la f_1 g_1, f_2 g_2 \ra_{Y_{m+n}} = \la g_1, \la f_1, f_2 \ra_{Y_m} \cdot g_2 \ra_{Y_n},
\end{equation}
and we see that $f \in Y_m$, $g \in Y_n \implies fg \in Y_{m+n}$ by taking $f=f_1=f_2 \in Y_m$, $g=g_1=g_2 \in Y_n$ in \autoref{eq: tps_inner_product}. It also follows from \autoref{eq: tps_inner_product} that $f \otimes g \mapsto fg$ extends to an isometric linear operator from $Y_m\otimes_{C_0(\Lambda^\infty)} Y_n$ to $Y_{m+n}$. To show that this map is surjective, we aim to apply \cite[Lemma~1.26]{Katsura2004TAMS}. For each $f \in Y_m$ and $g \in Y_n$, we define $f*g \colon \Lambda^\infty \to \C$ by $(f*g)(x) \coloneqq f(x)g(T^m(x))$. Then $f*g \in C_c(\Lambda^\infty)$ for each $f,g \in C_c(\Lambda^\infty)$. We define 
\[
\AA \coloneqq \vecspan\{ f*g : f \in C_c(\Lambda^\infty) \subseteq Y_m,\, g \in C_c(\Lambda^\infty) \subseteq Y_n \},
\]
and
\[
\BB \coloneqq \vecspan\{fg : f \in C_c(\Lambda^\infty) \subseteq Y_m,\, g \in C_c(\Lambda^\infty) \subseteq Y_n \}.
\]
Fix an open subset $U \subseteq \Lambda^\infty$. The set $\BB$ is obviously closed under addition (but due to the presence of the cocycle, is not obviously closed under multiplication), and we claim that $\BB \cap C_c(U)$ is uniformly dense in $C_c(U)$. To see this, first note that the set $\AA$ is easily seen to be a subalgebra of $C_0(\Lambda^\infty)$, and that a standard argument using the Stone--Weierstrass theorem shows that $\AA \cap C_c(U)$ is uniformly dense in $C_c(U)$. Now let $F \in C_c(U)$ and $\varepsilon > 0$. The map $G \colon x \mapsto \overline{c(x(0,m),x(m,m+n))} F(x)$ is continuous, and has the same support of $F$, so is an element of $C_c(U)$. We can then find $f_1,\dotsc,f_p \in C_c(\Lambda^\infty) \subseteq Y_m$ and $g_1,\dotsc,g_p \in C_c(\Lambda^\infty) \subseteq Y_n$ such that $\sum_{i=1}^p f_i*g_i \in \AA \cap C_c(U)$ and $\lv G - \sum_{i=1}^p f_i*g_i \rv_{\infty} < \varepsilon$. It follows that 
\[
\big\lv F- \sum_{i=1}^p f_ig_i \big\rv_{\infty} = \big\lv G - \sum_{i=1}^p f_i*g_i \big\rv_{\infty} < \varepsilon,
\]
and hence the claim holds. We can now apply \cite[Lemma~1.26]{Katsura2004TAMS} to see that $f\otimes g \mapsto fg$ is surjective. Hence it is an isomorphism. The other identities required to make $Y$ a product system follow from standard calculations.
\end{proof}

\begin{prop} \label{prop:compactlyaligned}
Suppose that $\Lambda$ is a proper, source-free topological $k$-graph, and $c$ is a continuous $\T$-valued $2$-cocycle on $\Lambda$. The product systems $X$ and $Y$ are compactly aligned. 
\end{prop}

\begin{proof}
We saw in \autoref{lem: top graphs from top k-graph} that the left actions $\phi_{X_n}$ and $\phi_{Y_n}$ are all by compacts, and so the result follows from \cite[Proposition~5.8]{Fowler2002}.
\end{proof}

\subsection{The twisted \texorpdfstring{$C^*$-algebras}{C*-algebras}} \label{subsec: the C*s}

\begin{definition} \label{def: the C*s}
Suppose that $\Lambda$ is a proper, source-free topological $k$-graph, and $c$ is a continuous $\T$-valued $2$-cocycle on $\Lambda$. We define the \emph{twisted Cuntz--Krieger algebra} $C^*(\Lambda,c)$ to be the Cuntz--Pimsner algebra $\OO(X)$. We define the \emph{twisted Toeplitz algebra} $\TT C^*(\Lambda,c)$ to be the Nica--Toeplitz algebra $\NT(X)$. 
\end{definition}

\begin{remarks} \label{rems: trivial twist and top} \leavevmode
\begin{enumerate}[label=(\roman*),ref=\autoref{rems: trivial twist and top}~(\roman*)]
\item When the twist is trivial, $C^*(\Lambda,c)$ and $\TT C^*(\Lambda,c)$ are precisely the $C^*$-algebras studied in \cite[Section~5.3]{CLSV2011} (for $\Lambda$ proper and source-free). We can then apply \cite[Theorem~5.20]{CLSV2011} to see that we obtain the groupoid $C^*$-algebras of topological $k$-graphs defined in \cite{Yeend2006CM}.

\item \label{item: generalising C*(Lambda,c)} Recall from \cite[Definition~5.2]{KPS2015TAMS} the definition of the twisted $C^*$-algebra of a row-finite $k$-graph $\Lambda$ with no sources, and a $\T$-valued categorical $2$-cocycle $c$ on $\Lambda$. This $C^*$-algebra is also denoted by $C^*(\Lambda,c)$, but for this remark we call it $C_{\KPS}^*(\Lambda,c)$. It is a straightforward exercise (using the gauge-invariant uniqueness theorem for injectivity) to prove that there is an isomorphism of $C_{\KPS}^*(\Lambda,c)$ onto $C^*(\Lambda,c)$ that sends each generating partial isometry $s_\lambda$ to $j_{X,d(\lambda)}(\delta_\lambda)$, where $\delta_\lambda \colon \Lambda^{d(\lambda)} \to \C$ is the usual point-mass function. So our twisted Cuntz--Krieger algebras generalise the twisted $C^*$-algebras of \cite{KPS2015TAMS} in the discrete setting.
\end{enumerate}
\end{remarks}

We now state our main theorem, which describes the relationship between the twisted Cuntz--Krieger and Toeplitz algebras, and the $C^*$-algebras associated to the product system $Y$ from \autoref{prop: infinite path product system}. We will prove this theorem in \autoref{sec: NC repn of X} and \autoref{sec: the C*s}.

\begin{thm} \label{thm: main theorem}
Suppose that $\Lambda$ is a proper, source-free topological $k$-graph, and $c$ is a continuous $\T$-valued $2$-cocycle on $\Lambda$. Let $X$ and $Y$ be the product systems from \autoref{prop: finite path product system} and \autoref{prop: infinite path product system}, respectively. Then there is a Nica-covariant representation $\psi\colon X\to \NT(Y)$ such that
\begin{enumerate}[label=(\alph*)]
\item \label{item: thm psi and the NT map} the induced homomorphism $\psi^{\NT}\colon \TT C^*(\Lambda,c)\to\NT(Y)$ is injective; and
\item \label{item: thm zeta and the O map} the map $\zeta\coloneqq q_Y\circ \psi\colon X\to\OO(Y)$ is a Cuntz--Pimsner-covariant representation of $X$, and the induced homomorphism $\zeta^\OO\colon C^*(\Lambda,c)\to\OO(Y)$ is an isomorphism.
\end{enumerate}
\end{thm}

\begin{remarks} \leavevmode
\begin{enumerate}[label=(\roman*)]
\begin{samepage}
\item \autoref{thm: main theorem} substantially generalises \cite[Proposition~8.6]{AaHR2018}, which applies in the untwisted setting, and to finite $1$-coaligned (discrete) $k$-graphs with no sources or sinks. (See also \cite[Theorem~7.1]{AaHR2014} in the finite directed graph setting.)
\end{samepage}

\item We do not expect the map $\psi^{\NT}$ from part (a) in Theorem~\ref{thm: main theorem} to be surjective, and we can use the results of \cite{aHLRS2014} and \cite{AaHR2018} on KMS states to see this. First observe that if $c$ is the trivial cocycle, and $\Lambda$ is a discrete $k$-graph, then $\TT C^*(\Lambda,c)$ is the $C^*$-algebra appearing in \cite[Corollary~7.5]{RS2005}. Now take $\Lambda$ to be a finite $1$-coaligned (discrete) $k$-graph with no sources or sinks, and $c$ to be the trivial cocycle. Then we see from the proof of \cite[Proposition~8.11]{AaHR2018} that a KMS state of $\TT C^*(\Lambda,c)$ is the restriction of many distinct KMS states on $\NT(Y)$. Hence the KMS simplex of $\NT(Y)$ is bigger than the KMS simplex of $\TT C^*(\Lambda,c)$, and so the two algebras are not isomorphic.
\end{enumerate}
\end{remarks}

\vspace{0.5em}
\section{A Nica-covariant representation of \texorpdfstring{$X$}{X} in \texorpdfstring{$\NT(Y)$}{NT(Y)}} \label{sec: NC repn of X}

In this section we build a Nica-covariant representation $\psi$ of the product system $X$ in the Nica--Toeplitz algebra of the product system $Y$. 

\begin{prop} \label{prop: NC repn of X}
Suppose that $\Lambda$ is a proper, source-free topological $k$-graph, and $c$ is a continuous $\T$-valued $2$-cocycle on $\Lambda$. Suppose that $X$ and $Y$ are the product systems from \autoref{prop: finite path product system} and \autoref{prop: infinite path product system}, respectively. 
\begin{enumerate}[label=(\alph*)]
\item \label{item: the alphas} For each $m,n \in \N^k$ with $m \ge n$, there is a map $\alpha_{n,m}\colon X_m \to Y_n$ given by $\alpha_{n,m}(f)(x) \coloneqq f(x(0,m))$, for all $f\in X_m$ and $x\in\Lambda^\infty$. We denote each $\alpha_{n,n}$ by $\alpha_n$.
\item \label{item: the NCR psi} The map $\psi\colon X\to \NT(Y)$ given by $\psi_n \coloneqq i_{Y,n} \circ \alpha_n$ is a Nica-covariant representation of $X$, where $\psi|_{X_n} \coloneqq \psi_n$.
\end{enumerate}
\end{prop}

\begin{remark}
For all $m,n,p \in \N^k$ with $m \ge n,p$ and each $f \in X_m$, we have
\[
\alpha_{n,m}(f)(x) = f(x(0,m)) = \alpha_{p,m}(f)(x),
\]
for all $x \in \Lambda^\infty$. However, $\alpha_{n,m}(f) \ne \alpha_{p,m}(f)$, because $\alpha_{n,m}(f) \in Y_n$, whereas $\alpha_{p,m}(f) \in Y_p$.
\end{remark}

\begin{proof}[Proof of \autoref{prop: NC repn of X}~\autoref{item: the alphas}]
Fix $m,n\in\N^k$ with $m\ge n$, and $f\in C_c(\Lambda^m)$. Recall from \autoref{lem: rho maps} that $\rho_{m,\infty}\colon \Lambda^\infty \to \Lambda^m$ is the continuous proper map given by $\rho_{m,\infty}(x) \coloneqq x(0,m)$. We have $\alpha_{n,m}(f) = f\circ\rho_{m,\infty}$, and hence $\alpha_{n,m}(f)$ is continuous. We have $\supp(\alpha_{n,m}(f)) \subseteq \rho_{m,\infty}^{-1}(\supp(f))$, which is a compact subset of $\Lambda^\infty$. Hence $\alpha_{n,m}(f) \in C_c(\Lambda^\infty)$. We have
\[
\lv f \rv_{X_m}^2 = \lv \la f,f \ra_{X_m} \rv = \sup\{\la f,f \ra_{X_m}(v) : v\in \Lambda^0\} = \sup\Big\{\sum_{\lambda\in \Lambda^mv} \lav f(\lambda) \rav^2 : v\in \Lambda^0\Big\},
\]
and
\begin{align*}
\lv \alpha_{n,m}(f) \rv_{Y_n}^2 & = \sup\Big\{ \sum_{\substack{y \in \Lambda^\infty, \\ T^n(y) = x}} \lav f(y(0,m)) \rav^2 : x \in \Lambda^\infty \Big\} \\
& = \sup\Big\{ \sum_{\mu\in\Lambda^nr(x)} \lav f(\mu x(0,m-n))\rav^2 : x \in \Lambda^\infty\Big \}.
\end{align*}
For each $x\in\Lambda^\infty$, we have
\[
\sum_{\mu\in\Lambda^nr(x)}\lav f(\mu x(0,m-n))\rav^2 \le \sum_{\lambda\in \Lambda^mx(m-n)} \lav f(\lambda)\rav^2,
\]
and it follows that $\lv \alpha_{n,m}(f) \rv_{Y_n} \le \lv f \rv_{X_m}$. Hence $\alpha_{n,m}$ is bounded, and extends to the desired map $X_m\to Y_n$.
\end{proof}

To prove \autoref{prop: NC repn of X}~\autoref{item: the NCR psi}, we need a number of results, including some technical calculations (see \autoref{lem: CLSV for X and Y}). We start with some properties of the $\alpha_n$ maps.

\begin{lemma} \label{lem: props of alpha}
Suppose that $\Lambda$ is a proper, source-free topological $k$-graph, and $c$ is a continuous $\T$-valued $2$-cocycle on $\Lambda$. Let $X$ and $Y$ be the product systems from \autoref{prop: finite path product system} and \autoref{prop: infinite path product system}, respectively. For each $m, n \in \N^k$, we have
\begin{enumerate}[label=(\roman*)]
\item \label{item: alpha respects left action} $\alpha_m(g \cdot f)=\alpha_0(g) \cdot \alpha_m(f)$, for all $f \in X_m$, $g\in C_0(\Lambda^0)$;
\item \label{item: alpha respects right action} $\alpha_m(f \cdot g)=\alpha_m(f) \cdot \alpha_0(g)$, for all $f \in X_m$, $g\in C_0(\Lambda^0)$;
\item \label{item: alpha respects ip} $\la \alpha_m(f), \alpha_m(g) \ra_{Y_m} = \alpha_0(\la f, g \ra_{X_m})$, for all $f,g \in X_m$;
\item \label{item: alpha is multi} $\alpha_{m+n}(fg)=\alpha_m(f)\alpha_n(g)$, for all $f \in X_m$, $g \in X_n$; and
\item \label{item: alpha_n is injective} $\alpha_n$ is injective, for each $n \in \N^k$.
\end{enumerate}
\end{lemma}

\begin{proof}
For \autoref{item: alpha is multi}, fix $x\in\Lambda^\infty$. Then
\begin{align*}
\alpha_{m+n}(fg)(x)&= (fg)(x(0,m+n)) \\
&= c(x(0,m),x(m,m+n))f(x(0,m))g(x(m,m+n)) \\
&= c(x(0,m),x(m,m+n))\alpha_m(f)(x)\alpha_n(g)(T^m(x)) \\
&= (\alpha_m(f)\alpha_n(g))(x),
\end{align*}
and so \autoref{item: alpha is multi} holds. Properties \autoref{item: alpha respects left action}--\autoref{item: alpha respects ip} follow from similarly straightforward calculations. For \autoref{item: alpha_n is injective}, suppose that $\alpha_n(f_1) = \alpha_n(f_2)$ for some $f_1, f_2 \in X_n$. Then for all $x \in \Lambda^\infty$, we have $f_1(x(0,n)) = f_2(x(0,n))$, and it follows by \autoref{rem: infinite paths coming into every vertex} that $f_1 = f_2$.
\end{proof}

Each $\alpha_n$ induces a homomorphism between the algebras $\KK(X_n)$ and $\KK(Y_n)$ of generalised compact operators.

\begin{lemma} \label{lem: the alpha KK map}
For each $n\in \N^k$, there is a homomorphism $\alpha_n^\KK \colon \KK(X_n)\to \KK(Y_n)$ satisfying $\alpha_n^\KK(\Theta_{f,g})=\Theta_{\alpha_n(f),\alpha_n(g)}$ for all $f,g\in X_n$.
\end{lemma}

\begin{proof}
Recall from \autoref{notation: F_E} that $F_{\Lambda_n}$ denotes the set of functions $f\in C_c(\Lambda^n)$ such that $\supp(f)$ is an $s|_{\Lambda^n}$-section, and $F_{\Lambda_{\infty,n}}$ denotes the set of functions $h \in C_c(\Lambda^\infty)$ such that $\supp(h)$ is a $T^n$-section. We know from \autoref{cor: new compacts} that $\KK(X_n) = \clspan\{\Theta_{f_1,f_2} : f_1,f_2\in F_{\Lambda_n}\}$, and $\KK(Y_n) = \clspan\{\Theta_{h_1,h_2} : h_1,h_2\in F_{\Lambda_{\infty,n}}\}$. 

We claim that the map $\sum_{i=1}^m\Theta_{f_i,g_i}\mapsto \sum_{i=1}^m\Theta_{\alpha_n(f_i),\alpha_n(g_i)}$, for $f_i,g_i\in F_{\Lambda_n}$, is norm decreasing. We have
\[
\Big\lv \sum_{i=1}^m \Theta_{\alpha_n(f_i),\alpha_n(g_i)} \Big\rv = \sup\Big\{ \Big\lav \Big( \sum_{i=1}^m \Theta_{\alpha_n(f_i),\alpha_n(g_i)} \Big)(h)(x)\Big\rav\, :\, h\in F_{\Lambda_{\infty,n}},\, \lv h\rv_{Y_n}\le 1,\, x\in \Lambda^\infty\Big\},
\]
and
\[
\Big\lv \sum_{i=1}^m \Theta_{f_i,g_i} \Big\rv = \sup\Big\{ \Big\lav \Big( \sum_{i=1}^m \Theta_{f_i,g_i} \Big)(l)(\mu)\Big\rav\, :\, l\in F_{\Lambda_n},\, \lv l\rv_{X_n}\le 1,\, \mu\in \Lambda^n\Big\}.
\]
Fix $h\in F_{\Lambda_{\infty,n}}$ with $\lv h \rv_{Y_n} \le 1$, and $x \in \Lambda^\infty$. Let $G\coloneqq \{ i : 1 \le i \le m,\, s|_{\supp(g_i)}^{-1}(x(n)) \ne \emptyset \}$. For all $i\in G$, let $\mu_i$ be the unique path in $s|_{\supp(g_i)}^{-1}(x(n)) \subseteq \Lambda^n$, and let $x_i\coloneqq \mu_iT^n(x)$. Then
\begin{align*}
\Big\lav \Big( \sum_{i=1}^m \Theta_{\alpha_n(f_i),\alpha_n(g_i)} \Big)(h)(x)\Big\rav &= \Big\lav \sum_{i=1}^m \Big(f_i(x(0,n))\sum_{ \substack{ y\in\Lambda^\infty, \\ T^n(y)=T^n(x)} }\overline{g_i(y(0,n))}h(y)\Big)\Big\rav \\
&= \Big\lav \sum_{i\in G} f_i(x(0,n))\overline{g_i(\mu_i)}h(x_i)\Big\rav \\
&= \lav f_j(x(0,n))\overline{g_j(\mu_j)}h(x_j) \rav,
\end{align*}
if $j$ is the unique element of $G$ such that $T^n|_{\supp(h)}^{-1}(T^n(x))=\{\mu_jT^n(x)\}$, and is zero if there is no such $j\in G$. We only have to deal with the former case, but in this case we choose $\mu\coloneqq x(0,n)$ and $l\in F_{\Lambda_n}$ with $\lv l\rv_{X_n} \le 1$ and $l(\mu_j)=1$. We then have 
\begin{align*}
\Big\lav \Big( \sum_{i=1}^m \Theta_{f_i,g_i} \Big)(l)(\mu)\Big\rav &= \Big\lav \sum_{i=1}^m \Big(f_i(\mu)\sum_{ \lambda\in\Lambda^ns(\mu) }\overline{g_i(\lambda)}l(\lambda)\Big)\Big\rav \\
&= \lav f_j(\mu)\overline{g_j(\mu_j)}l(\mu_j) \rav \\
&= \lav f_j(x(0,n))\overline{g_j(\mu_j)} \rav \\
&\ge \lav f_j(x(0,n))\overline{g_j(\mu_j)}h(x_j) \rav \\
&= \Big\lav \Big( \sum_{i=1}^m \Theta_{\alpha_n(f_i),\alpha_n(g_i)} \Big)(h)(x)\Big\rav .
\end{align*}
Therefore, we have 
\[
\Big\lv \sum_{i=1}^m \Theta_{\alpha_n(f_i),\alpha_n(g_i)} \Big\rv \le \Big\lv \sum_{i=1}^m \Theta_{f_i,g_i} \Big\rv,
\]
and the claim is proved. It follows that $\alpha_n^\KK(\sum_{i=1}^m\Theta_{f_i,g_i})\coloneqq \sum_{i=1}^m\Theta_{\alpha_n(f_i),\alpha_n(g_i)}$ extends to a linear map $\alpha_n^\KK \colon \KK(X_n)\to \KK(Y_n)$ satisfying $\alpha_n^\KK(\Theta_{f,g})=\Theta_{\alpha_n(f),\alpha_n(g)}$ for all $f,g\in X_n$. We now check that $\alpha_n^\KK$ is a homomorphism, using the identities in \autoref{lem: props of alpha}:
\begin{align*}
\alpha_n^\KK(\Theta_{f_1,f_2}\Theta_{g_1,g_2}) &= \alpha_n^\KK(\Theta_{(f_1\cdot\la f_2,g_1\ra_{X_n}),g_2}) \\ 
&= \Theta_{(\alpha_n(f_1)\cdot\alpha_0(\la f_2,g_1\ra_{X_n})),\alpha_n(g_2)} \\
&= \Theta_{(\alpha_n(f_1)\cdot \la\alpha_n(f_2),\alpha_n(g_1)\ra_{Y_n}),\alpha_n(g_2)} \\
&=\Theta_{\alpha_n(f_1),\alpha_n(f_2)}\Theta_{\alpha_n(g_1),\alpha_n(g_2)} \\
&=\alpha_n^\KK(\Theta_{f_1,f_2})\alpha_n^\KK(\Theta_{g_1,g_2}),
\end{align*}
for all $f_1,f_2,g_1,g_2\in X_n$.
\end{proof}

The key to proving that $\psi$ is Nica covariant is the following technical lemma. 

\begin{lemma} \label{lem: CLSV for X and Y}
Suppose that $\Lambda$ is a proper, source-free topological $k$-graph, and $c$ is a continuous $\T$-valued $2$-cocycle on $\Lambda$. Suppose that $X$ and $Y$ are the product systems from \autoref{prop: finite path product system} and \autoref{prop: infinite path product system}, respectively. Let $m,n \in \N^k$. For $i\in\{1,2\}$, let $f_i \in F_{\Lambda_m}$ and $g_i \in F_{\Lambda_n}$. Define $C\coloneqq \supp(f_2)\vee \supp(g_1)$, which is a compact subset of $\Lambda^{m\vee n}$. For each $p \in \{m,n\}$, let $\{V_i^p : 1\le i\le r_p\}$ be a finite open cover of $\tau_{p, m \vee n}(C) \subseteq \Lambda^{(m \vee n) - p}$ such that each $V_i^p$ is a precompact $s$-section. 
\begin{enumerate}[label=(\alph*)]
\item \label{item: iota-thetas in X} Let $\beta_1^p,\dotsc,\beta_{r_p}^p$ be a partition of unity subordinate to $\{V_i^p\cap \tau_{p,m\vee n}(C) : 1\le i\le r_p\}$, and fix functions $\gamma_i^p\colon \Lambda^{(m\vee n)-p}\to [0,1]$ such that each $\gamma_i^p|_{\tau_{p,m\vee n}(C)} = \sqrt{\beta_i^p}$ and each $\gamma_i^p$ vanishes off $V_i^p$. For $i \in \{1,\dotsc,r_m\}$ and $j \in \{1,\dotsc,r_n\}$, define $a_{ij},b_j\in C_c(\Lambda^{m \vee n}) \subseteq X_{m \vee n}$ by $a_{ij}\coloneqq f_1\big(\gamma_i^m\cdot\la f_2\gamma_i^m, g_1\gamma_j^n\ra_{X_{m\vee n}}\big)$, and $b_j\coloneqq g_2\gamma_j^n$.
Then
\begin{equation} \label{eq: iota-thetas in X}
\iota_m^{m\vee n}(\Theta_{f_1,f_2})\iota_n^{m\vee n}(\Theta_{g_1,g_2})=\sum_{i=1}^{r_m}\sum_{j=1}^{r_n} \Theta_{a_{ij},b_j}.
\end{equation}
\item \label{item: iota-thetas in Y} Let $\eta_1^p,\dotsc,\eta_{r_p}^p$ be a partition of unity subordinate to $\{Z(V_i^p)\cap Z(\tau_{p,m\vee n}(C)) : 1\le i\le r_p\}$, and fix functions $\xi_i^p\colon \Lambda^\infty \to [0,1]$ such that each $\xi_i^p|_{Z(\tau_{p,m\vee n}(C))} = \sqrt{\eta_i^p}$ and each $\xi_i^p$ vanishes off $Z(V_i^p)$. For $i \in \{1,\dotsc,r_m\}$ and $j \in \{1,\dotsc,r_n\}$, define $c_{ij},d_j\in C_c(\Lambda^\infty) \subseteq Y_{m \vee n}$ by $c_{ij}\coloneqq \alpha_m(f_1)\big(\xi_i^m\cdot\la \alpha_m(f_2)\xi_i^m, \alpha_n(g_1)\xi_j^n\ra_{Y_{m\vee n}}\big)$, and $d_j\coloneqq \alpha_n(g_2)\xi_j^n$.
Then
\begin{equation} \label{eq: iota-thetas in Y}
\iota_m^{m\vee n}(\Theta_{\alpha_m(f_1),\alpha_m(f_2)})\iota_n^{m\vee n}(\Theta_{\alpha_n(g_1),\alpha_n(g_2)})=\sum_{i=1}^{r_m}\sum_{j=1}^{r_n} \Theta_{c_{ij},d_j}.
\end{equation}
\end{enumerate}
\end{lemma}

\autoref{eq: iota-thetas in X} appears for the ``untwisted'' version of $X$ in \cite[Lemma~5.14]{CLSV2011}, and the proofs of \autoref{eq: iota-thetas in X} and \autoref{eq: iota-thetas in Y} follow in a similar way. We therefore only give a detailed proof of \autoref{lem: CLSV for X and Y}~\autoref{item: iota-thetas in Y}. We start with some notation (see \cite[Notation~5.10]{CLSV2011}) and a preparatory lemma (see \cite[Lemma~5.12]{CLSV2011}).

\begin{notation}
Let $m\in \N^k$ and $f\in F_{\Lambda_m}$. Then for each $v\in \Lambda^0$, the set $\osupp(f)\cap \Lambda^m v$ is either empty or consists of a single path. In the latter case, we will denote this path by $\lambda_{f,v}$.
\end{notation}

\begin{lemma} \label{lem: CLSV5.12 for Y}
Suppose that $\Lambda$ is a proper, source-free topological $k$-graph, $c$ is a continuous $\T$-valued $2$-cocycle on $\Lambda$, and $Y$ is the product system from \autoref{prop: infinite path product system}. Let $m,n\in \N^k$, and $f,g \in F_{\Lambda_m}$. Then for each $h\in Y_{m\vee n}$ and $z\in \Lambda^\infty$, we have
\begin{align*}
&\iota_m^{m\vee n}(\Theta_{\alpha_m(f),\alpha_m(g)})(h)(z) \\
& \hspace{1.5em}= c(z(0,m),z(m,m\vee n))\overline{ c(\lambda_{g,z(m)},z(m,m\vee n)) }f(z(0,m))\overline{g(\lambda_{g,z(m)})}h(\lambda_{g,z(m)}T^m(z)),
\end{align*}
if $\supp(g) \cap \Lambda^mz(m) \ne \emptyset$; otherwise, $\iota_m^{m\vee n}(\Theta_{\alpha_m(f),\alpha_m(g)})(h)(z) = 0$.
\end{lemma}

\begin{proof}
For $h_1\in Y_m$, $h_2\in Y_{(m\vee n)-m}$, and $z\in \Lambda^\infty$, we have
\begin{align*}
& \iota_m^{m\vee n}(\Theta_{\alpha_m(f),\alpha_m(g)})(h_1h_2)(z) \\
& \hspace{1.5em} = c(z(0,m),z(m,m\vee n))f(z(0,m))\Big(\sum_{ \substack{ y\in \Lambda^\infty, \\ T^m(y)=T^m(z) }}\overline{g(y(0,m))}h_1(y)\Big)h_2(T^m(z)).
\end{align*}
If $\supp(g) \cap \Lambda^mz(m) \ne \emptyset$, then one term that will appear in the sum is $\overline{ g(\lambda_{g,z(m)})}h_1(\lambda_{g,z(m)}T^m(z))$, which corresponds to taking $y=\lambda_{g,z(m)}T^m(z)$. All other terms in this sum will be zero, since $\supp(g)$ is an $s$-section. Hence we can continue the calculation to get
\begin{align*}
&\iota_m^{m\vee n}(\Theta_{\alpha_m(f),\alpha_m(g)})(h_1h_2)(z) \\
& \hspace{0.3em} = c(z(0,m),z(m,m\vee n))f(z(0,m))\overline{ g(\lambda_{g,z(m)})}h_1(\lambda_{g,z(m)}T^m(z))h_2(T^m(z)) \\
& \hspace{0.3em} = c(z(0,m),z(m,m\vee n))f(z(0,m))\overline{ g(\lambda_{g,z(m)})}\overline{ c(\lambda_{g,z(m)},z(m,m\vee n)) }(h_1h_2)(\lambda_{g,z(m)}T^m(z)).
\end{align*}
On the other hand, if $\supp(g) \cap \Lambda^mz(m) = \emptyset$, then $\iota_m^{m\vee n}(\Theta_{\alpha_m(f),\alpha_m(g)})(h_1h_2)(z) = 0$. Since $Y_{m\vee n}=\clspan\{h_1h_2 : h_1\in Y_m,\, h_2\in Y_{(m\vee n) - m}\}$, the result follows.
\end{proof}

\begin{proof}[Proof of \autoref{lem: CLSV for X and Y}~\autoref{item: iota-thetas in Y}]
Fix $h \in Y_{m \vee n}$ and $z \in \Lambda^\infty$. The right-hand side of \autoref{eq: iota-thetas in Y} is given by
\begin{align*}
\Big(\sum_{i=1}^{r_m}\sum_{j=1}^{r_n} \Theta_{c_{ij},d_j}\Big)(h)(z) &= \sum_{i=1}^{r_m} \sum_{j=1}^{r_n} c_{ij}(z) \la d_j,h \ra_{Y_{m\vee n}}(T^{m\vee n}(z)) \\
&= \sum_{j=1}^{r_n} \Bigg( c(z(0,m),z(m,m\vee n))f_1(z(0,m)) \la d_j,h \ra_{Y_{m\vee n}}(T^{m\vee n}(z)) \\ 
&\hspace{7em} \Big(\sum_{i=1}^{r_m} \xi_i^m \cdot \la \alpha_m(f_2)\xi_i^m , \alpha_n(g_1)\xi_j^n \ra_{Y_{m\vee n}}\Big)(T^m(z))\Bigg) \numberthis \label{eq: calc0 for rhs of CLSV for Y} .
\end{align*}
Consider the expression $\big(\sum_{i=1}^{r_m}\xi_i^m\cdot\la \alpha_m(f_2)\xi_i^m, \alpha_n(g_1)\xi_j^n\ra_{Y_{m\vee n}}\big)(T^m(z))$ from above. We have 
\begin{align*}
& \Big(\sum_{i=1}^{r_m}\xi_i^m\cdot\la \alpha_m(f_2)\xi_i^m, \alpha_n(g_1)\xi_j^n\ra_{Y_{m\vee n}}\Big)(T^m(z)) \\
&\hspace{0.1em}= \sum_{i=1}^{r_m}\Big(\xi_i^m(T^m(z)) \sum_{ \substack{y\in \Lambda^\infty, \\ T^{m\vee n}(y)=T^{m\vee n}(z)} } \overline{(\alpha_m(f_2)\xi_i^m)(y)}(\alpha_n(g_1)\xi_j^n)(y)\Big) \\
&\hspace{0.1em}= \sum_{i=1}^{r_m}\Big(\xi_i^m(T^m(z)) \sum_{ \substack{y\in \Lambda^\infty, \\ T^{m\vee n}(y)=T^{m\vee n}(z)} } \overline{ c(y(0,m),y(m,m\vee n))f_2(y(0,m))\xi_i^m(T^m(y)) }(\alpha_n(g_1)\xi_j^n)(y)\Big).
\end{align*}
Since each $\xi_i^m$ is supported on $Z(V_i^m)$, we have $z(m, m \vee n) \in V_i^m$. If $y \in \Lambda^\infty$ with $T^m(y) \in \osupp(\xi_i^m)$ and $T^{m \vee n}(y) = T^{m \vee n}(z)$, then since $V_i^m$ is an $s$-section, we have $y(m, m \vee n) = z(m, m \vee n)$, and so $T^m(y) = T^m(z)$. Now, if $\osupp(f_2) \cap \Lambda^m z(m) \ne \emptyset$, then one term that will appear in the second sum of the final line of the above calculation is
\[
\overline{ c(\lambda_{f_2,z(m)},z(m,m\vee n))f_2(\lambda_{f_2,z(m)})\xi_i^m(T^m(z))}(\alpha_n(g_1)\xi_j^n)( \lambda_{f_2,z(m)}T^m(z)),
\]
which corresponds to taking $y=\lambda_{f_2,z(m)}T^m(z)$. All other terms in this sum will be zero, since $\supp(f_2)$ and $V_i^m$ are $s$-sections. Hence we can continue the calculation to get
\begin{align*}
\Big(\sum_{i=1}^{r_m}\xi_i^m\cdot&\la \alpha_m(f_2)\xi_i^m,\, \alpha_n(g_1)\xi_j^n\ra_{Y_{m\vee n}}\Big)(T^m(z)) \\
&= \sum_{i=1}^{r_m}(\xi_i^m(T^m(z)))^2\overline{ c(\lambda_{f_2,z(m)},z(m,m\vee n))f_2(\lambda_{f_2,z(m)})}(\alpha_n(g_1)\xi_j^n)( \lambda_{f_2,z(m)}T^m(z)) \\
&= \overline{ c(\lambda_{f_2,z(m)},z(m,m \vee n))f_2(\lambda_{f_2,z(m)})}(\alpha_n(g_1)\xi_j^n)( \lambda_{f_2,z(m)}T^m(z)). \numberthis \label{eq: calc1 for rhs of CLSV for Y}
\end{align*}
On the other hand, if $\osupp(f_2) \cap \Lambda^m z(m) = \emptyset$, then we have
\[
\Big(\sum_{i=1}^{r_m}\xi_i^m\cdot\la \alpha_m(f_2)\xi_i^m, \alpha_n(g_1)\xi_j^n\ra_{Y_{m\vee n}}\Big)(T^m(z)) = 0.
\]
We now substitute \autoref{eq: calc1 for rhs of CLSV for Y} into \autoref{eq: calc0 for rhs of CLSV for Y} to see that if $\osupp(f_2) \cap \Lambda^m z(m) \ne \emptyset$, then we have
\begin{align*}
\Big(\sum_{i=1}^{r_m}\sum_{j=1}^{r_n} \Theta_{c_{ij},d_j}\Big)(h)(z) = \sum_{j=1}^{r_n} \Big(&c(z(0,m),z(m,m\vee n)) f_1(z(0,m)) \la d_j,h \ra_{Y_{m\vee n}}(T^{m\vee n}(z)) \\ 
&\overline{ c(\lambda_{f_2,z(m)},z(m,m\vee n))f_2(\lambda_{f_2,z(m)})}(\alpha_n(g_1)\xi_j^n)( \lambda_{f_2,z(m)}T^m(z))\Big) \numberthis \label{eq: another step}.
\end{align*}
We define 
\[
w \coloneqq \lambda_{f_2,z(m)} T^m(z)\in\Lambda^\infty,\quad \beta \coloneqq w(n, m \vee n) \in \Lambda^{(m\vee n)-n},\ \text{ and }\ \, \delta \coloneqq z(m, m \vee n) \in \Lambda^{(m\vee n)-m}.
\]
Calculating $(\alpha_n(g_1)\xi_j^n)( \lambda_{f_2,z(m)}T^m(z))$ and factoring out terms that do not depend on $j$ in \autoref{eq: another step}, we see that
\begin{equation} \label{eq: factoring out non-j terms}
\Big(\sum_{i=1}^{r_m}\sum_{j=1}^{r_n} \Theta_{c_{ij},d_j}\Big)(h)(z) = \Xi_1 \Big( \sum_{j=1}^{r_n}\xi_j^n(T^n(w))\la d_j,h\ra_{Y_{m\vee n}}(T^{m\vee n}(z)) \Big),
\end{equation}
where 
\[
\Xi_1 \coloneqq c(z(0,m),\delta) \overline{ c(\lambda_{f_2,z(m)},\delta	)} c(w(0,n),\beta)f_1(z(0,m)) \overline{ f_2(\lambda_{f_2,z(m)})} g_1(w(0,n)).
\]
Since $T^{m \vee n}(z) = T^{m \vee n}(w)$, we have
\begin{align*}
\sum_{j=1}^{r_n}\xi_j^n(T^n&(w))\la d_j,h\ra_{Y_{m\vee n}}(T^{m\vee n}(z)) \\
&= \sum_{j=1}^{r_n}\Big(\xi_j^n(T^n(w))\sum_{\substack{ y\in\Lambda^\infty, \\ T^{m\vee n}(y)=T^{m\vee n}(w) }}\overline{ c(y(0,n),y(n,m\vee n))g_2(y(0,n))\xi_j^n(T^n(y))}h(y)\Big).
\end{align*}
Since each $\xi_j^n$ is supported on $Z(V_j^n)$, we have $w(n, m \vee n) \in V_j^n$. If $y \in \Lambda^\infty$ with $T^n(y) \in \osupp(\xi_j^n)$ and $T^{m \vee n}(y) = T^{m \vee n}(w)$, then since $V_j^n$ is an $s$-section, we have $y(n, m \vee n) = w(n, m \vee n) = \beta$, and so $T^n(y) = T^n(w)$. Now, if $\osupp(g_2) \cap \Lambda^n r(\beta) \ne \emptyset$, then one term that will appear in the second sum of the right-hand side of the above equation is
\[
\overline{c(\lambda_{g_2,r(\beta)},\beta)g_2(\lambda_{g_2,r(\beta)})\xi_j^n(T^n(w))} h(\lambda_{g_2,r(\beta)}T^n(w)),
\]
which corresponds to taking $y=\lambda_{g_2,r(\beta)}T^n(w)$. All other terms in this sum will be zero, since $\supp(g_2)$ and $V_j^n$ are $s$-sections. Hence we can continue the calculation to get
\begin{align*}
\sum_{j=1}^{r_n}\xi_j^n(T^n(w))\la d_j,h\ra_{Y_{m\vee n}}(T^{m\vee n}(z)) &= \sum_{j=1}^{r_n}(\xi_j^n(T^n(w)))^2\overline{c(\lambda_{g_2,r(\beta)},\beta)g_2(\lambda_{g_2,r(\beta)})} h(\lambda_{g_2,r(\beta)}T^n(w)) \\
&= \overline{c(\lambda_{g_2,r(\beta)},\beta)g_2(\lambda_{g_2,r(\beta)})} h(\lambda_{g_2,r(\beta)}T^n(w)). \numberthis \label{eq: calc2 for rhs of CLSV for Y}
\end{align*}
On the other hand, if $\osupp(g_2) \cap \Lambda^n r(\beta) = \emptyset$, then we have
\[
\sum_{j=1}^{r_n}\xi_j^n(T^n(w))\la d_j,h\ra_{Y_{m\vee n}}(T^{m\vee n}(z)) = 0.
\]
We can now use \autoref{eq: factoring out non-j terms} and \autoref{eq: calc2 for rhs of CLSV for Y} to conclude that, when $\osupp(f_2) \cap \Lambda^m z(m) \ne \emptyset$ and $\osupp(g_2) \cap \Lambda^n r(\beta) \ne \emptyset$, we have  
\begin{equation} \label{eq: RHS = Xi_2 times stuff}
\Big(\sum_{i=1}^{r_m}\sum_{j=1}^{r_n} \Theta_{c_{ij},d_j}\Big)(h)(z) = \Xi_2\, h(\lambda_{g_2,r(\beta)}T^n(w)),
\end{equation}
where
\[
\Xi_2 \coloneqq \Xi_1\,\overline{c(\lambda_{g_2,r(\beta)},\beta)g_2(\lambda_{g_2,r(\beta)})}.
\]

For the left-hand side of \autoref{eq: iota-thetas in Y}, we first apply \autoref{lem: CLSV5.12 for Y} with $f=f_1$, $g=f_2$ to see that $\iota_m^{m\vee n}(\Theta_{\alpha_m(f_1),\alpha_m(f_2)})\Big(\iota_n^{m\vee n}(\Theta_{\alpha_n(g_1),\alpha_n(g_2)})(h)\Big)(z)=0$ when $\osupp(f_2) \cap \Lambda^m z(m) = \emptyset$, and is otherwise given by
\[
\iota_m^{m\vee n}(\Theta_{\alpha_m(f_1),\alpha_m(f_2)})\Big(\iota_n^{m\vee n}(\Theta_{\alpha_n(g_1),\alpha_n(g_2)})(h)\Big)(z) = \Xi_3\, \iota_n^{m\vee n}(\Theta_{\alpha_n(g_1),\alpha_n(g_2)})(h)(\lambda_{f_2,z(m)} T^m(z)),
\]
where
\[
\Xi_3 \coloneqq c(z(0,m),\delta)\overline{c(\lambda_{f_2,z(m)},\delta)} f_1(z(0,m)) \overline{ f_2(\lambda_{f_2,z(m)}) }.
\]
We now apply \autoref{lem: CLSV5.12 for Y} with $f=g_1$ and $g=g_2$. Recall that $w = \lambda_{f_2,z(m)} T^m(z)$ and $\beta = w(n, m \vee n)$. If $\osupp(g_2) \cap \Lambda^n r(\beta) = \emptyset$, then $\iota_n^{m\vee n}(\Theta_{\alpha_n(g_1),\alpha_n(g_2)})(h)(w)=0$. If $\osupp(g_2) \cap \Lambda^n r(\beta) \ne \emptyset$, then we have
\[
\iota_n^{m\vee n}(\Theta_{\alpha_n(g_1),\alpha_n(g_2)})(h)(w) = \Xi_4\, h(\lambda_{g_2,r(\beta)}T^n(w)),
\]
where
\[
\Xi_4 \coloneqq c(w(0,n),\beta)\overline{c(\lambda_{g_2,r(\beta)},\beta)}g_1(w(0,m))\overline{g_2(\lambda_{g_2,r(\beta)})}.
\]
Hence we have
\begin{equation} \label{eq: LHS = Xi_2 times stuff}
\iota_m^{m\vee n}(\Theta_{\alpha_m(f_1),\alpha_m(f_2)})\Big(\iota_n^{m\vee n}(\Theta_{\alpha_n(g_1),\alpha_n(g_2)})(h)\Big)(z)= \Xi_3\,\Xi_4\, h(\lambda_{g_2,r(\beta)}T^n(w)).
\end{equation}
By comparing \autoref{eq: RHS = Xi_2 times stuff} with \autoref{eq: LHS = Xi_2 times stuff}, we see that $\Xi_3\,\Xi_4 = \Xi_2$, and hence both sides of \autoref{eq: iota-thetas in Y} agree.
\end{proof}

\begin{lemma} \label{lem: identity with alphas and Thetas}
Suppose that $\Lambda$ is a proper, source-free topological $k$-graph, and $c$ is a continuous $\T$-valued $2$-cocycle on $\Lambda$. Suppose that $X$ and $Y$ are the product systems from \autoref{prop: finite path product system} and \autoref{prop: infinite path product system}, respectively. Let $m,n \in \N^k$. For $i \in \{1,2\}$, let $f_i \in F_{\Lambda_m}$ and $g_i \in F_{\Lambda_n}$. Then we have
\[
\alpha_{m\vee n}^\KK\big(\iota_m^{m\vee n}(\Theta_{f_1,f_2})\iota_n^{m\vee n}(\Theta_{g_1,g_2})\big) = \iota_m^{m\vee n}(\Theta_{\alpha_m(f_1),\alpha_m(f_2)})\iota_n^{m\vee n}(\Theta_{\alpha_n(g_1),\alpha_n(g_2)}).
\] 
\end{lemma}

\begin{proof}
Fix such functions $f_1,f_2, g_1$, and $g_2$, and define $C\coloneqq \supp(f_2)\vee \supp(g_1)$. For $p \in \{m,n\}$ and $i \in \{1,\dotsc,r_p\}$, let $V_i^p$, $\beta_i^p$, and $\gamma_i^p$ be as in \autoref{lem: CLSV for X and Y}~\autoref{item: iota-thetas in X}. Then we have
\[
\iota_m^{m\vee n}(\Theta_{f_1,f_2})\iota_n^{m\vee n}(\Theta_{g_1,g_2})=\sum_{i=1}^{r_m}\sum_{j=1}^{r_n} \Theta_{a_{ij},b_j}.
\]

Using \autoref{item: alpha respects right action}, \autoref{item: alpha respects ip}, and \autoref{item: alpha is multi} of \autoref{lem: props of alpha}, we get, for each $i \in \{1,\dotsc,r_m\}$ and $j \in \{1,\dotsc,r_n\}$,
\begin{align*}
\alpha_{m\vee n}(a_{ij}) &= \alpha_m(f_1)\alpha_{(m\vee n)-m}\big(\gamma_i^m\cdot\la f_2\gamma_i^m, g_1\gamma_j^n\ra_{X_{m \vee n}}\big) \\
&= \alpha_m(f_1)\big(\alpha_{(m\vee n)-m}(\gamma_i^m)\cdot \alpha_0\big( \la f_2\gamma_i^m, g_1\gamma_j^n\ra_{X_{m \vee n}}\big)\big) \\
&= \alpha_m(f_1)\big(\alpha_{(m\vee n)-m}(\gamma_i^m)\cdot \la \alpha_{m\vee n}(f_2\gamma_i^m), \alpha_{m\vee n}(g_1\gamma_j^n)\ra_{Y_{m \vee n}}\big) \\
&= \alpha_m(f_1)\big(\alpha_{(m\vee n)-m}(\gamma_i^m)\cdot \la \alpha_m(f_2)
\alpha_{(m\vee n)-m}(\gamma_i^m), \alpha_n(g_1)\alpha_{(m\vee n)-n}(\gamma_j^n)\ra_{Y_{m \vee n}}\big), \numberthis \label{eq: the alpha aij}
\end{align*}
and 
\begin{equation} \label{eq: the alpha bj}
\alpha_{m\vee n}(b_j)=\alpha_n(g_2)\alpha_{(m\vee n)-n}(\gamma_j^n).
\end{equation}
We claim that for each $p \in \{m,n\}$ and $i \in \{1,\dotsc,r_p\}$, the functions 
\[
\eta_i^p\coloneqq \alpha_{(m\vee n)-p}(\beta_i^p)|_{Z(\tau_{p,m\vee n}(C))}
\]
and $\xi_i^p\coloneqq \alpha_{(m\vee n)-p}(\gamma_i^p)$ satisfy the assumptions of \autoref{lem: CLSV for X and Y}~\autoref{item: iota-thetas in Y}. For each $x\in Z(\tau_{p,m\vee n}(C)) \setminus Z(V_i^p)$, we have $\eta_i^p(x) = \beta_i^p(x(0,(m\vee n)-p))=0$. For all $x\in Z(\tau_{p,m\vee n}(C))$, we have $\big(\sum_{i=1}^{r_p}\eta_i^p\big)(x)=\big(\sum_{i=1}^{r_p} \beta_i^p\big)(x(0,(m\vee n)-p)) = 1$. It follows that $\eta_1^p,\dotsc,\eta_{r_p}^p$ is a partition of unity subordinate to $\{Z(V_i^p)\cap Z(\tau_{p,m\vee n}(C)) : 1\le i\le r_p\}$. The functions $\xi_j^p\colon \Lambda^\infty \to [0,1]$ satisfy
\begin{align*}
\xi_i^p|_{Z(\tau_{p,m\vee n}(C))} &= \alpha_{(m\vee n)-p}(\gamma_i^p)|_{Z(\tau_{p,m\vee n}(C))} \\ 
&= \big(\alpha_{(m\vee n)-p}(\gamma_i^p|_{\tau_{p,m\vee n}(C)})\big)|_{Z(\tau_{p,m\vee n}(C))} \\ 
&= \big(\alpha_{(m\vee n)-p}\big(\sqrt{\beta_i^p}\big)\big)|_{Z(\tau_{p,m\vee n}(C))} \\
&= \sqrt{\alpha_{(m\vee n)-p}(\beta_i^p)|_{Z(\tau_{p,m\vee n}(C))}} \\
&= \sqrt{\eta_i^p},
\end{align*}
and vanish off $Z(V_i^p)$ because $\gamma_i^p$ vanishes off $V_i^p$. So the claim holds. We see from \autoref{eq: the alpha aij} and \autoref{eq: the alpha bj} that the $c_{ij}$ and $d_j$ from \autoref{lem: CLSV for X and Y}~\autoref{item: iota-thetas in Y} are given by $c_{ij}=\alpha_{m\vee n}(a_{ij})$ and $d_j=\alpha_{m\vee n}(b_j)$. We can then apply \autoref{lem: CLSV for X and Y} to get
\begin{align*}
\alpha_{m\vee n}^\KK\big(\iota_m^{m\vee n}(\Theta_{f_1,f_2})\iota_n^{m\vee n}(\Theta_{g_1,g_2})\big) &= \alpha_{m\vee n}^\KK\Big(\sum_{i=1}^{r_m}\sum_{j=1}^{r_n} \Theta_{a_{ij},b_j}\Big) \\
&= \sum_{i=1}^{r_m}\sum_{j=1}^{r_n}\Theta_{\alpha_{m\vee n}(a_{ij}),\alpha_{m\vee n}(b_j)} \\
&= \sum_{i=1}^{r_m}\sum_{j=1}^{r_n}\Theta_{c_{ij},d_j} \\
&= \iota_m^{m\vee n}(\Theta_{\alpha_m(f_1),\alpha_m(f_2)})\iota_n^{m\vee n}(\Theta_{\alpha_n(g_1),\alpha_n(g_2)}).\qedhere
\end{align*}
\end{proof}

\begin{proof}[Proof of \autoref{prop: NC repn of X}~\autoref{item: the NCR psi}]
We first claim that $\psi\colon X\to \NT(Y)$ given by $\psi_n \coloneqq i_{Y,n}\circ\alpha_n$ is a representation of $X$.The map $\psi_0$ is obviously a homomorphism of $C_0(\Lambda^0)$, and for each $n \in \N^k$, it is easy to see that the map $\psi_n$ is linear. To see that each $(\psi_n,\psi_0)$ is a representation of $X_n$, fix $g\in C_0(\Lambda^0)$ and $f\in X_n$. \autoref{lem: props of alpha}~\autoref{item: alpha respects left action} gives
\[
\psi_n(g\cdot f)=i_{Y,n}(\alpha_n(g\cdot f))=i_{Y,n}(\alpha_0(g)\cdot\alpha_n(f))=i_{Y,0}(\alpha_0(g))i_{Y,n}(\alpha_n(f))=\psi_0(g)\psi_n(f),
\]
and \autoref{lem: props of alpha}~\autoref{item: alpha respects right action} gives
\[
\psi_n(f\cdot g)=i_{Y,n}(\alpha_n(f\cdot g))=i_{Y,n}(\alpha_n(f)\cdot\alpha_0(g))=i_{Y,n}(\alpha_n(f))i_{Y,0}(\alpha_0(g))=\psi_n(f)\psi_0(g).
\]
Now fix $f_1,f_2\in X_n$. \autoref{lem: props of alpha}~\autoref{item: alpha respects ip} gives
\begin{align*}
\psi_n(f_1)^*\psi_n(f_2) &= i_{Y,n}(\alpha_n(f_1))^*i_{Y,n}(\alpha_n(f_2)) \\
&= i_{Y,0}(\la \alpha_n(f_1),\alpha_n(f_2)\ra_{Y_n}) \\
&= i_{Y,0}(\alpha_0(\la f_1,f_2\ra_{X_n})) \\
&= \psi_0(\la f_1,f_2\ra_{X_n}).
\end{align*}
Hence each $(\psi_n,\psi_0)$ is a representation of $X_n$. For each $m,n\in \N^k$, $f\in X_m$, and $g\in X_n$ we use \autoref{lem: props of alpha}~\autoref{item: alpha is multi} and that $i_Y$ is a representation to get
\[
\psi_{m+n}(fg)=i_{Y,m+n}(\alpha_{m+n}(fg))=i_{Y,m+n}(\alpha_m(f)\alpha_n(g))=i_{Y,m}(\alpha_m(f))i_{Y,n}(\alpha_n(g))=\psi_m(f)\psi_n(g).
\]
Hence $\psi\colon X\to \NT(Y)$ is a representation. 

To see that $\psi$ is Nica covariant, first note that for each $n\in\N^k$ and $f,g\in X_n$, we have
\[
\psi^{(n)}(\Theta_{f,g}) = \psi_n(f)\psi_n(g)^* = i_{Y,n}(\alpha_n(f))i_{Y,n}(\alpha_n(g))^* = i_Y^{(n)}(\Theta_{\alpha_n(f),\alpha_n(g)}) = i_Y^{(n)}(\alpha_n^\KK(\Theta_{f,g})).
\]
It follows that $\psi^{(n)} = i_Y^{(n)} \circ \alpha_n^\KK$. For each $m,n\in\N^k$, $f_1,f_2\in F_{\Lambda_m}$, and $g_1,g_2\in F_{\Lambda_n}$, we now use the Nica covariance of $i_Y$, and \autoref{lem: identity with alphas and Thetas}, to get
\begin{align*}
\psi^{(m)}(\Theta_{f_1,f_2})\psi^{(n)}(\Theta_{g_1,g_2}) &= i_Y^{(m)}(\alpha_m^\KK(\Theta_{f_1,f_2})) i_Y^{(n)}(\alpha_n^\KK(\Theta_{g_1,g_2})) \\
&= i_Y^{(m)}(\Theta_{\alpha_m(f_1),\alpha_m(f_2)}) i_Y^{(n)}(\Theta_{\alpha_n(g_1),\alpha_n(g_2)}) \\
&= i_Y^{(m\vee n)}\big( \iota_m^{m\vee n}(\Theta_{\alpha_m(f_1),\alpha_m(f_2)})\iota_n^{m\vee n}(\Theta_{\alpha_n(g_1),\alpha_n(g_2)}) \big) \\
&= i_Y^{(m\vee n)}\big( \alpha_{m\vee n}^\KK\big(\iota_m^{m\vee n}(\Theta_{f_1,f_2})\iota_n^{m\vee n}(\Theta_{g_1,g_2})\big) \big) \\
&= \psi^{(m\vee n)}(\iota_m^{m\vee n}(\Theta_{f_1,f_2})\iota_n^{m\vee n}(\Theta_{g_1,g_2})) \numberthis \label{eq: Nica cov for nice Thetas}.
\end{align*}
Now, let $S\in \KK(X_m)$ and $T\in\KK(X_n)$, and fix $\varepsilon > 0$. If $T=0$, then 
\[
\psi^{(m)}(S)\psi^{(n)}(T) = 0 = \psi^{(m\vee n)}(\iota_m^{m\vee n}(S)\iota_n^{m\vee n}(T)).
\]
So assume that $T\not =0$. We use \autoref{cor: new compacts} to choose a nonzero $a\in \vecspan\{\Theta_{f,g} : f,g \in F_{\Lambda_m}\}$ such that $\lv S-a\rv < \varepsilon/(4\lv T\rv)$, and $b\in\vecspan \{\Theta_{f,g}: f,g \in F_{\Lambda_n}\}$ such that $\lv T-b\rv < \varepsilon/(4\lv a\rv)$. Then
\begin{align*}
& \lv \psi^{(m)}(S)\psi^{(n)}(T) - \psi^{(m\vee n)}(\iota_m^{m\vee n}(S)\iota_n^{m\vee n}(T)) \rv \\
& \hspace{2em} = \Big\lv \psi^{(m)}(S)\psi^{(n)}(T) - \psi^{(m)}(a)\psi^{(n)}(T) + \psi^{(m)}(a)\psi^{(n)}(T) - \psi^{(m)}(a)\psi^{(n)}(b) \\
& \hspace{4em} + \psi^{(m)}(a)\psi^{(n)}(b) - \psi^{(m\vee n)}( \iota_m^{m\vee n}(a)\iota_n^{m\vee n}(b) ) \\
& \hspace{4em} + \psi^{(m\vee n)}\big( \iota_m^{m\vee n}(a)\iota_n^{m\vee n}(b) - \iota_m^{m\vee n}(a)\iota_n^{m\vee n}(T) + \iota_m^{m\vee n}(a)\iota_n^{m\vee n}(T) -\iota_m^{m\vee n}(S)\iota_n^{m\vee n}(T) \big) \Big\rv \\
& \hspace{2em} \le \lv \psi^{(m)}(S)\psi^{(n)}(T) - \psi^{(m)}(a)\psi^{(n)}(T)\rv + \lv \psi^{(m)}(a)\psi^{(n)}(T) - \psi^{(m)}(a)\psi^{(n)}(b)\rv \\
& \hspace{4em} + \lv \psi^{(m)}(a)\psi^{(n)}(b) - \psi^{(m\vee n)}( \iota_m^{m\vee n}(a)\iota_n^{m\vee n}(b) )\rv \\
& \hspace{4em} + \lv \iota_m^{m\vee n}(a)\iota_n^{m\vee n}(b) - \iota_m^{m\vee n}(a)\iota_n^{m\vee n}(T)\rv + \lv \iota_m^{m\vee n}(a)\iota_n^{m\vee n}(T) -\iota_m^{m\vee n}(S)\iota_n^{m\vee n}(T)\rv \\
& \hspace{2em} \le 2\lv T\rv\, \lv S-a\rv + 2 \lv a\rv\, \lv T - b\rv + \lv \psi^{(m)}(a)\psi^{(n)}(b) - \psi^{(m\vee n)}( \iota_m^{m\vee n}(a)\iota_n^{m\vee n}(b) )\rv \\
& \hspace{2em} < \epsilon + \lv \psi^{(m)}(a)\psi^{(n)}(b) - \psi^{(m\vee n)}( \iota_m^{m\vee n}(a)\iota_n^{m\vee n}(b) )\rv.
\end{align*}
It follows from \autoref{eq: Nica cov for nice Thetas} and the linearity of the maps involved that 
\[
\lv \psi^{(m)}(a)\psi^{(n)}(b) - \psi^{(m\vee n)}( \iota_m^{m\vee n}(a)\iota_n^{m\vee n}(b) )\rv=0.
\] 
Hence we have 
\[
\lv \psi^{(m)}(S)\psi^{(n)}(T) - \psi^{(m\vee n)}(\iota_m^{m\vee n}(S)\iota_n^{m\vee n}(T)) \rv < \varepsilon,
\]
and the Nica covariance of $\psi$ follows.
\end{proof}

\vspace{0.5em}
\section{The proof of our main theorem} \label{sec: the C*s}

We begin this section with the proof of \autoref{thm: main theorem}~\autoref{item: thm psi and the NT map}, which states that the homomorphism $\psi^{\NT}\colon \TT C^*(\Lambda,c) \to \NT(Y)$ induced from the Nica-covariant representation $\psi\colon X\to \NT(Y)$ given in \autoref{prop: NC repn of X}~\autoref{item: the NCR psi} is injective.

\begin{proof}[Proof of \autoref{thm: main theorem}~\autoref{item: thm psi and the NT map}]
Recall that $\psi^\NT$ satisfies $\psi^\NT\circ i_X=\psi$. To show that $\psi^\NT$ is injective, we aim to apply \autoref{thm: Fletcher uniqueness theorem}. We consider the Nica-covariant representation $\Psi$ from the proof of \cite[Theorem~3.2]{Fletcher2017} (see \autoref{thm: Fletcher uniqueness theorem}) applied to the product system $Y$; that is, 
\[
\Psi\coloneqq \Big( \FF_Y{-}\Ind_{C_0(\Lambda^\infty)}^{\LL(\FF_Y)}\pi \Big)\circ l,
\]
where $l$ is the Fock representation of $Y$ on the Fock space $\FF_Y\coloneqq \bigoplus_{n\in \N^k} Y_n$, $\pi$ is a faithful nondegenerate representation of $C_0(\Lambda^\infty)$ on some Hilbert space $\HH$, and $ \FF_Y{-}\Ind_{C_0(\Lambda^\infty)}^{\LL(\FF_Y)}\pi$ is the induced representation of $\LL(\FF_Y)$ on $\FF_Y\otimes_{C_0(\Lambda^\infty)}\HH$. In the proof of \cite[Theorem~3.2]{Fletcher2017}, it is shown that for each finite $K\subseteq \N^k {\setminus} \{0\}$ the representation $\varphi_{\Psi,K}$ of $C_0(\Lambda^\infty)$ on $\FF_Y\otimes_{C_0(\Lambda^\infty)}\HH$ given by
\[
\varphi_{\Psi,K}(f) \coloneqq \Psi_0(f)\prod_{n\in K} \big( 1_{\FF_Y\otimes_{C_0(\Lambda^\infty)}\HH} - P_n^\Psi)
\]
is faithful. Fix a finite subset $K\subseteq \N^k {\setminus\{0\}}$, and consider the representation 
\[
\omega\coloneqq \Psi^{\NT}\circ\psi \colon X\to B( \FF_Y\otimes_{C_0(\Lambda^\infty)}\HH).
\] 
We have 
\[
\omega_n = \Psi^\NT\circ \psi_n = \Psi^\NT\circ i_{Y,n}\circ \alpha_n = \Psi_n\circ \alpha_n,
\]
for each $n\in\N^k$. So each $\overline{\omega_n(X_n)(\FF_Y\otimes_{C_0(\Lambda^\infty)}\HH)}$ is a subspace of $\overline{\Psi_n(Y_n)(\FF_Y\otimes_{C_0(\Lambda^\infty)}\HH)}$, and it follows that
\begin{equation} \label{eq: defect projections}
\prod_{n\in K} \big( 1_{\FF_Y\otimes_{C_0(\Lambda^\infty)}\HH} - P_n^\omega) \ge \prod_{n\in K} \big( 1_{\FF_Y\otimes_{C_0(\Lambda^\infty)}\HH} - P_n^\Psi).
\end{equation}
For each nonzero $f\in C_0(\Lambda^0)$, we know from \autoref{lem: props of alpha}~\autoref{item: alpha_n is injective} that $\alpha_0(f) \ne 0$, and hence
\[
\omega_0(f) \prod_{n\in K} \big( 1_{\FF_Y\otimes_{C_0(\Lambda^\infty)}\HH} - P_n^\Psi) = \Psi_0(\alpha_0(f)) \prod_{n\in K} \big( 1_{\FF_Y\otimes_{C_0(\Lambda^\infty)}\HH} - P_n^\Psi)=\varphi_{\Psi,K}(\alpha_0(f)) \ne 0.
\]
It follows from \autoref{eq: defect projections} that
\[
\omega_0(f) \prod_{n\in K} \big( 1_{\FF_Y\otimes_{C_0(\Lambda^\infty)}\HH} - P_n^\omega) \ne 0.
\]
Hence 
\[
\varphi_{\omega,K}\colon f\mapsto \omega_0(f)\prod_{n\in K} \big( 1_{\FF_Y\otimes_{C_0(\Lambda^\infty)}\HH} - P_n^\omega)
\]
is a faithful representation of $C_0(\Lambda^0)$, and we can apply \autoref{thm: Fletcher uniqueness theorem} to see that
\[
\omega^\NT\colon \TT C^*(\Lambda,c) \to B(\FF_Y\otimes_{C_0(\Lambda^\infty)}\HH)
\]
is injective. Now, 
\[
\Psi^\NT\circ \psi^\NT \circ i_X = \Psi^\NT\circ \psi = \omega = \omega^\NT\circ i_X, 
\]
and hence $\omega^\NT = \Psi^\NT \circ \psi^\NT$. Since $\omega^\NT$ is injective, $\psi^\NT$ must also be injective.
\end{proof}

We need the following lemma for the proof of \autoref{thm: main theorem}~\autoref{item: thm zeta and the O map}. In this result, we identify $C_0(Z(U))$ with its image under the natural embedding in $C_0(\Lambda^\infty)$. For the statement and proof of this result, recall from \autoref{lem: tau maps} that each $\tau_{m,n}$ is a local homeomorphism, and from \autoref{lem: rho maps} that each of the maps $\rho_{m,n}$ and $\rho_{m,\infty}$ is continuous and proper.

\begin{lemma} \label{lem: C_0(Z(U)) SWT}
Suppose that $\Lambda$ is a proper, source-free topological $k$-graph. Fix $n\in \N^k$, and let $U$ be an open subset of $\Lambda^n$ such that $\overline{U}$ is a compact $s$-section. Define
\begin{align*}
\AA_U \coloneqq \{ f\circ\rho_{m,\infty} : m \in \N^k,~&m \ge n, f \in C_c(\Lambda^m) \text{ with } \osupp(f)\subseteq UV \text{ for } V \subseteq \Lambda^{m-n} \\
&\text{ a precompact open $s$-section such that } r(\overline{V}) \subseteq s(U) \}.
\end{align*}
Then $\vecspan\AA_U$ is uniformly dense in $C_0(Z(U))$.
\end{lemma}

\begin{proof}
We will use the Stone--Weierstrass theorem. We first need to show that $\vecspan\AA_U$ is a subalgebra of $C_0(Z(U))$. Let $m,p\in \N^k$ with $m,p\ge n$; $f\in C_c(\Lambda^m)$ with $\osupp(f)\subseteq UV_m$ for $V_m \subseteq \Lambda^{m-n}$ a precompact open $s$-section such that $r(\overline{V_m}) \subseteq s(U)$; and $g\in C_c(\Lambda^p)$ with $\osupp(g)\subseteq UV_p$ for $V_p \subseteq \Lambda^{p-n}$ a precompact open $s$-section such that $r(\overline{V_p}) \subseteq s(U)$. Then we have
\[
\osupp(f \circ \rho_{m,\infty}) = \rho_{m,\infty}^{-1}(\osupp(f)) \subseteq \rho_{m,\infty}^{-1}(UV_m) = Z(UV_m) \subseteq Z(U).
\]
Also, since $\rho_{m,\infty}$ is a continuous proper map and $\supp(f \circ \rho_{m,\infty}) \subseteq \rho_{m,\infty}^{-1}(\supp(f))$, we have $f \circ \rho_{m,\infty} \in C_c(\Lambda^\infty)$. Hence $\vecspan\AA_U \subseteq C_0(Z(U))$.

We now claim that $(f\circ \rho_{m,\infty})(g\circ \rho_{p,\infty}) \in \vecspan\AA_U$. Since composition is a continuous map and $\overline{U}$, $\overline{V_m}$, and $\overline{V_p}$ are compact, $\overline{U}\,\overline{V_m}$ and $\overline{U}\,\overline{V_p}$ are closed subsets of $\Lambda^m$ and $\Lambda^p$, respectively. Hence $\supp(f) \subseteq \overline{U}\,\overline{V_m}$ and $\supp{g} \subseteq \overline{U}\,\overline{V_p}$. The set $C\coloneqq \supp(f)\vee \supp(g)$ is a compact subset of $\Lambda^{m \vee p}$ and is contained in $\overline{U}\,\overline{V_m} \vee \overline{U}\,\overline{V_p}$. We therefore have
\[
r(\tau_{n,m \vee p}(C)) \subseteq r(\tau_{n,m \vee p}(\overline{U}\,\overline{V_m} \vee \overline{U}\,\overline{V_p})) \subseteq r(\overline{V_m}) \cap r(\overline{V_p}) \subseteq s(U),
\]
and so $\tau_{n,m \vee p}(C)$ is a compact subset of $r|_{\Lambda^{(m \vee p)-n}}^{-1}(s(U))$. Hence we can cover $\tau_{n,m\vee p}(C)$ with finitely many precompact open $s$-sections $V_1,\dotsc,V_l \subseteq \Lambda^{(m \vee p)-n}$ such that $r(\overline{V_i}) \subseteq s(U)$ for each $i \in \{1,\dotsc,l\}$. As in \cite[Remark~2.9]{LPS2014}, let $\xi_1, \dotsc, \xi_l$ be a partition of unity subordinate to $\{V_i \cap \tau_{n, m \vee p}(C) : 1 \le i \le l\}$. Define $h\colon \Lambda^{m\vee p}\to \C$ by $h(\lambda) \coloneqq f(\lambda(0,m))g(\lambda(0,p))$. We have $h=(f\circ\rho_{m,m\vee p})(g\circ \rho_{p,m\vee p})$, and hence $h$ is continuous. Furthermore, we have
\[
\supp(h) \subseteq \supp(f \circ \rho_{m,m \vee p}) \subseteq \rho_{m,m \vee p}^{-1}(\supp(f)),
\]
and so $\supp(h)$ is compact because $\rho_{m,m \vee p}$ is a continuous proper map. For each $i \in \{1,\dotsc,l\}$, define $h_i\colon \Lambda^{m\vee p}\to \C$ by 
\[
h_i(\lambda) \coloneqq
\begin{cases}
h(\lambda)(\xi_i\circ \tau_{n,m\vee p})(\lambda) & \text{if } \lambda \in C, \\
0 & \text{otherwise.}
\end{cases}
\]
Then $\supp(h_i) \subseteq \supp(h)$, and so $h_i\in C_c(\Lambda^{m\vee p})$ for each $i \in \{1,\dotsc,l\}$. We also have 
\begin{align*}
\osupp(h_i) &= \osupp(h)\, \cap\, \osupp(\xi_i \circ \tau_{n,m \vee p}|_C) \\
&= \osupp(f \circ \rho_{m,m \vee p})\, \cap\, \osupp(g \circ \rho_{p,m \vee p})\, \cap\, (\tau_{n,m \vee p}|_C)^{-1}(\osupp(\xi_i)) \\
&= \rho_{m,m \vee p}^{-1}(\osupp(f)) \cap \rho_{p,m \vee p}^{-1}(\osupp(g))\, \cap\, \tau_{n,m \vee p}^{-1}(\osupp(\xi_i)) \cap C \\
&\subseteq \rho_{m,m \vee p}^{-1}(UV_m)\, \cap\, \rho_{p,m \vee p}^{-1}(UV_p)\, \cap\, \tau_{n,m \vee p}^{-1}(V_i) \\
&= U V_m \Lambda^{(m \vee p)-m}\, \cap\, U V_p \Lambda^{(m \vee p)-p}\, \cap\, \Lambda^n V_i \\
&\subseteq UV_i,
\end{align*}
and so $h_i \in \AA_U$ for each $i \in \{1,\dotsc,l\}$.

Furthermore, for each $x \in \Lambda^\infty$ we have
\begin{align*}
\Big(\sum_{i=1}^l h_i\circ\rho_{m \vee p,\infty}\Big)(x) &= \sum_{i=1}^l h_i(x(0,m\vee p)) \\
&= \sum_{i=1}^l h(x(0,m\vee p))\xi_i(x(n,m\vee p)) \\
&= f(x(0,m))g(x(0,p))\sum_{i=1}^l\xi_i(x(n,m\vee p)) \\
&=
\begin{cases}
f(x(0,m))g(x(0,p)) & \text{if } x \in Z(C), \\
0 & \text{otherwise}
\end{cases} \\
&= (f\circ\rho_{m,\infty})(g\circ\rho_{p,\infty})(x).
\end{align*}
Hence $(f\circ\rho_{m,\infty})(g\circ\rho_{p,\infty})=\sum_{i=1}^l h_i\circ\rho_{m \vee p,\infty} \in \vecspan\AA_U$, and the claim is proved.

It is clear that $\vecspan\AA_U$ is closed under complex conjugation. To see that $\vecspan\AA_U$ strongly separates the points of $Z(U)$, let $x,y\in Z(U)$ with $x \ne y$. Choose $m \in \N^k$ with $m \ge n$ and $x(0,m) \ne y(0,m)$. If $x(0,n) = y(0,n)$, then we must have $x(n,m) \ne y(n,m)$. If $x(0,n) \ne y(0,n)$, then $x(n) \ne y(n)$ because $s|_U$ is injective, and so $x(n,m) \ne y(n,m)$. Since $\Lambda^{m-n}$ is a locally compact Hausdorff space, there exist disjoint precompact open $s$-sections $V_x$ and $V_y$ contained in $\Lambda^{m-n}$ such that $r(\overline{V_x}) \subseteq s(U)$, $r(\overline{V_y}) \subseteq s(U)$, $x(n,m) \in V_x$, and $y(n,m) \in V_y$. By Urysohn's lemma, there exist $f_x, f_y \in C_c(\Lambda^m)$ such that $0 \le f_x, f_y \le 1$, $\osupp(f_x) \subseteq UV_x$, $\osupp(f_y) \subseteq UV_y$, and $f_x(x(0,m)) = f_y(y(0,m)) = 1$. Since $V_x \cap V_y = \emptyset$, we have $y(n,m) \notin V_x$, and hence $y(0,m) \notin UV_x$. Hence $f_x(y(0,m)) = 0 \ne 1 = f_x(x(0,m))$. The Stone--Weierstrass theorem now implies that $\clspan\AA_U = C_0(Z(U))$.
\end{proof}

We need one final lemma before we can prove \autoref{thm: main theorem}~\autoref{item: thm zeta and the O map}.

\begin{lemma} \label{lem: uniform norm is Y_n norm}
Suppose that $\Lambda$ is a proper, source-free topological $k$-graph. Fix $n \in \N^k$, and let $U \subseteq \Lambda^n$ be a precompact open $s$-section. Then for all $f \in C_c(Z(U))$, we have $\lv f \rv_\infty = \lv f \rv_{Y_n}$.
\end{lemma}

\begin{proof}
Since $U \subseteq \Lambda^n$ is an $s$-section, \autoref{prop: shift maps} implies that $T^n|_{Z(U)}$ is injective. Therefore, we have
\[
\lv f \rv_{Y_n}^2 = \lv \la f, f \ra_{Y_n} \rv = \sup \Big\{ \Big\lav \sum_{\substack{y \in \Lambda^\infty, \\ T^n(y) = x}} \lav f(y) \rav^2\, \Big\rav : x \in \Lambda^\infty \Big\} = \sup \{ \lav f(y) \rav^2 : y \in Z(U) \} = \lv f \rv_\infty^2.
\]
\end{proof}

We can now finish the proof of our main result.

\begin{proof}[Proof of \autoref{thm: main theorem}~\autoref{item: thm zeta and the O map}]
Consider the representation $\zeta\coloneqq q_Y\circ\psi\colon X\to \OO(Y)$. To see that $\zeta$ is Cuntz--Pimsner covariant, we first claim that for each $n\in\N^k$ we have
\begin{equation} \label{eq: for CP covariance of zeta}
\psi^{(n)}\circ\phi_{X_n}-\psi_0 = (i_Y^{(n)}\circ\phi_{Y_n}-i_{Y,0})\circ\alpha_0.
\end{equation}
Fix $n \in \N^k$ and $g\in C_c(\Lambda^0)$. Since $r|_{\Lambda^n}$ is proper, $N \coloneqq r|_{\Lambda^n}^{-1}(\supp(g))$ is a compact subset of $\Lambda^n$. We can therefore cover $N$ with finitely many open $s$-sections $U_1, \dotsc, U_l \subseteq \Lambda^n$. Choose a partition of unity $\widetilde{\omega}_1,\dotsc,\widetilde{\omega}_l$ subordinate to $\{U_i \cap N : 1 \le i \le l\}$. We use the Tietze extension theorem to extend each $\widetilde{\omega}_i$ to a function $\omega_i \in C_c(\Lambda^n)$. Note that, whilst $\omega_1, \dotsc, \omega_l$ is not a partition of unity on $\Lambda^n$, we do have $\big( \sum_{i=1}^l \omega_i\big)|_N \equiv 1$, and $\osupp(\omega_i|_N) \subseteq U_i$ for each $i \in \{1,\dotsc,l\}$. For each $i \in \{1,\dotsc,l\}$, we define $g_i \coloneqq \sqrt{g \cdot \omega_i}$. Then each $g_i$ is an element of $C_c(\Lambda^n) \subseteq X_n$. Now, for each $f \in Y_n$ and $\lambda \in \Lambda^n$, we have
\[
\Big( \sum_{i=1}^l \Theta_{g_i, \overline{g_i}} \Big)(f)(\lambda) = \sum_{i=1}^l \sqrt{g(r(\lambda))}\sqrt{\omega_i(\lambda)}\Big(\sum_{\mu\in \Lambda^n s(\lambda)}\sqrt{g(r(\mu))}\sqrt{\omega_i(\mu)}f(\mu)\Big).
\]
Now, for this expression to be nonzero, we must have $\lambda,\mu \in U_i \cap N$. Since $U_i$ is an $s$-section, the only nonzero term that could appear in the right-hand sum is the one obtained by taking $\mu = \lambda$. It follows that 
\[
\Big( \sum_{i=1}^l \Theta_{g_i, \overline{g_i}} \Big)(f)(\lambda) = \sum_{i=1}^l g(r(\lambda))\omega_i(\lambda) f(\lambda) = g(r(\lambda))f(\lambda)=\phi_{X_n}(g)(f)(\lambda).
\]
Hence we have
\begin{equation} \label{eq: left action in X}
\phi_{X_n}(g)=\sum_{i=1}^l \Theta_{g_i, \overline{g_i}}.
\end{equation} 
For each $h\in Y_n$ and $x\in \Lambda^\infty$, we have
\begin{align*}
\Big( \sum_{i=1}^l \Theta_{\alpha_n(g_i), \alpha_n(\overline{g_i})} \Big)(h)(x) &= \sum_{i=1}^l \alpha_n(g_i)(x)\Bigg(\sum_{\substack{y\in\Lambda^\infty, \\ T^n(y)=T^n(x)}}\overline{\alpha_n(\overline{g_i})(y)}h(y)\Bigg) \\
&=\sum_{i=1}^l g_i(x(0,n))\Bigg(\sum_{\substack{y\in\Lambda^\infty, \\ T^n(y)=T^n(x)}}g_i(y(0,n)h(y)\Bigg).
\end{align*}
Now, for this expression to be nonzero, we must have $x(0,n), y(0,n) \in U_i \cap N$. Since $T^n$ is injective on $U_i \cap N$, the only term that could appear in the right-hand sum is the one obtained by taking $y = x$. It follows that
\begin{align*}
\Big( \sum_{i=1}^l \Theta_{\alpha_n(g_i), \alpha_n(\overline{g_i})} \Big)(h)(x) 
&= \sum_{i=1}^l (g_i(x(0,n)))^2 h(x)\\
&= g(r(x(0,n)))h(x) \\
&= \alpha_0(g)(T^0(x))h(x) \\
&=\phi_{Y_n}(\alpha_0(g))(h)(x).
\end{align*}
Hence we have
\begin{equation} \label{eq: left action in Y}
\phi_{Y_n}(\alpha_0(g))=\sum_{i=1}^l \Theta_{\alpha_n(g_i), \alpha_n(\overline{g_i})} = \alpha_n^\KK\Big(\sum_{i=1}^l \Theta_{g_i,\overline{g_i}} \Big).
\end{equation}
We see from \autoref{eq: left action in X} and \autoref{eq: left action in Y} that $\alpha_n^\KK \circ \phi_{X_n} = \phi_{Y_n}\circ \alpha_0$. Hence we have
\[
\psi^{(n)}\circ\phi_{X_n}-\psi_0 = i_Y^{(n)}\circ \alpha_n^\KK \circ \phi_{X_n} - i_{Y,0}\circ \alpha_0 = i_Y^{(n)} \circ \phi_{Y_n}\circ \alpha_0 - i_{Y,0}\circ \alpha_0 = (i_Y^{(n)}\circ\phi_{Y_n}-i_{Y,0})\circ\alpha_0,
\]
and the claim holds. We now use \autoref{eq: for CP covariance of zeta} and the Cuntz--Pimsner covariance of $j_Y$ to get 
\begin{align*}
\zeta^{(n)}\circ\phi_{X_n}-\zeta_0 &= (q_Y\circ \psi^{(n)})\circ\phi_{X_n} - q_Y\circ \psi_0 \\ 
&= q_Y\circ (\psi^{(n)}\circ\phi_{X_n}-\psi_0) \\ 
&= q_Y\circ \big((i_Y^{(n)}\circ\phi_{Y_n}-i_{Y,0})\circ\alpha_0\big) \\
&= (j_Y^{(n)}\circ\phi_{Y_n}-j_{Y,0})\circ\alpha_0 \\
&\equiv 0.
\end{align*}
Hence $\zeta$ is Cuntz--Pimsner covariant, and so it induces a homomorphism $\zeta^\OO\colon C^*(\Lambda,c) \to \OO(Y)$ satisfying $\zeta^\OO\circ j_X=\zeta$. We now show that $\zeta^\OO$ is an isomorphism. To see that $\zeta^\OO$ is surjective, fix $g \in C_c(\Lambda^\infty)$ and $n \in \N^k$. We aim to show that $j_{Y,n}(g) \in \range(\zeta^\OO)$. 

Choose $m \in \N^k$ with $m \ge n$. Let $A$ and $U$ be open $s$-sections of $\Lambda$ such that $\overline{U}$ is compact, and $\overline{U} \subseteq A \subseteq \Lambda^n$. Let $B$ and $V$ be open $s$-sections of $\Lambda$ such that $\overline{V}$ is compact, $\overline{V} \subseteq B \subseteq \Lambda^{m-n}$, and $r(\overline{V}) \subseteq s(U)$. For each $v \in r(\overline{V})$, let $\lambda_{U,v}$ denote the unique element of $U$ with source $v$. Note that if $\lambda \in A$ and $\mu \in s(\lambda)\overline{V}$, then $\lambda_{U,r(\mu)} = \lambda$, because $U$ is contained in the $s$-section $A$. Let $f\in C_c(\Lambda^m)$ with $\osupp(f)\subseteq UV$. Recall from \autoref{prop: NC repn of X} the definitions of the maps $\alpha_{n,m}$ and $\alpha_n \coloneqq \alpha_{n,n}$. We claim that 
\begin{align} \label{eq: j_Y(f) in range zeta^OO}
j_{Y,n}(\alpha_{n,m}(f)) \in \range(\zeta^\OO).
\end{align}

The map $V\to \C$ given by $\mu\mapsto f(\lambda_{U,r(\mu)}\mu)$ is continuous, and has compact support because its support is contained in $\tau_{n,m}(\supp(f))$. Since $V$ is open, we can extend this function to $\widetilde{f}\in C_c(\Lambda^{m-n})$ given by
\[
\widetilde{f}(\mu)=
\begin{cases}
f(\lambda_{U,r(\mu)}\mu) & \ \text{if } \mu \in V, \\
0 & \ \text{if } \mu \not\in V.
\end{cases}
\]
Since $\overline{U}$ is a compact subset of $\Lambda^n$, we can cover it with finitely many precompact open $s$-sections $W_1,\dotsc,W_l \subseteq \Lambda^n$. Choose a partition of unity $\widetilde{\xi_1},\dotsc,\widetilde{\xi_l}$ subordinate to $\{W_i \cap \overline{U} : 1 \le i \le l\}$. Use the Tietze extension theorem to extend each $\widetilde{\xi_i}$ to a function $\xi_i$ in $C_c(\Lambda^n)$ with $\osupp(\xi_i) \subseteq A$. Note that, whilst $\xi_1,\dotsc,\xi_l$ is not a partition of unity on $\Lambda^n$, we do have $\big(\sum_{i=1}^l \xi_i\big)|_{\overline{U}} \equiv 1$, and $\osupp(\xi_i|_{\overline{U}}) \subseteq W_i$ for each $i \in \{1,\dotsc,l\}$. For each $x\in \Lambda^\infty$, we have
\begin{align*}
\Big(\sum_{i=1}^l \alpha_n(\xi_i)\cdot\alpha_{0,m-n}(\widetilde{f})\Big)(x) &= \sum_{i=1}^l \alpha_n(\xi_i)(x)\alpha_{0,m-n}(\widetilde{f})(T^n(x)) \\
&= \sum_{i=1}^l \xi_i(x(0,n))\widetilde{f}(x(n,m)) \\
&= 
\begin{cases}
\widetilde{f}(x(n,m)) & \ \text{if } x \in Z(\overline{U}), \\
0 & \ \text{otherwise}
\end{cases} \\
&= f(x(0,m)) \\
&= \alpha_{n,m}(f)(x).
\end{align*}
Hence we have
\begin{equation} \label{eq: action decomp for alpha}
\alpha_{n,m}(f)=\sum_{i=1}^l\alpha_n(\xi_i)\cdot\alpha_{0,m-n}(\widetilde{f}).
\end{equation}
Since $\overline{V}$ is a compact subset of $\Lambda^{m-n}$, we can cover it with finitely many precompact open $s$-sections $Z_1,\dotsc,Z_p \subseteq \Lambda^{m-n}$. Choose a partition of unity $\widetilde{\eta}_1,\dotsc,\widetilde{\eta}_p$ subordinate to $\{Z_j \cap \overline{V} : 1 \le j \le p\}$. Use the Tietze extension theorem to extend each $\widetilde{\eta}_i$ to a function $\eta_i \in C_c(\Lambda^{m-n})$ with $\osupp(\eta_i) \subseteq B$.  Note that, whilst $\eta_1,\dotsc,\eta_p$ is not a partition of unity on $\Lambda^{m-n}$, we do have $\big(\sum_{j=1}^p \eta_j\big)|_{\overline{V}} \equiv 1$, and $\osupp(\eta_j|_{\overline{V}}) \subseteq Z_j$ for each $j \in \{1,\dotsc,p\}$. For each $g\in Y_{m-n}$ and $x\in\Lambda^\infty$, we have
\begin{align*}
\Big(\sum_{j=1}^p \Theta_{\alpha_{m-n}(\widetilde{f}),\alpha_{m-n}(\eta_j)}\Big)(g)(x) &= \sum_{j=1}^p\alpha_{m-n}(\widetilde{f})(x)\la\alpha_{m-n}(\eta_j),g\ra_{Y_{m-n}}(T^{m-n}(x)) \\
&= \alpha_{m-n}(\widetilde{f})(x)\sum_{j=1}^p\sum_{\substack{y\in\Lambda^\infty, \\ T^{m-n}(y) = T^{m-n}(x)}}\overline{\alpha_{m-n}(\eta_j)(y)}g(y).
\end{align*}
Now, we know that $\alpha_{m-n}(\widetilde{f})$ is supported in $Z(\overline{V})$, and each $\alpha_{m-n}(\eta_j)$ is supported in $Z(B)$. Since $B$ is an $s$-section, \autoref{prop: shift maps} implies that $T^{m-n}|_{Z(B)}$ is injective. Thus, if $x \in Z(\overline{V}) \subseteq Z(B)$, then we have $\{y\in Z(B):T^{m-n}(y)=T^{m-n}(x)\}=\{x\}$, and so we can continue the above calculation to get 
\begin{align*}
\Big(\sum_{j=1}^p \Theta_{\alpha_{m-n}(\widetilde{f}),\alpha_{m-n}(\eta_j)}\Big)(g)(x) &= \alpha_{m-n}(\widetilde{f})(x)\Big(\sum_{j=1}^p\overline{\alpha_{m-n}(\eta_j)(x)}g(x)\Big) \\
&= \alpha_{m-n}(\widetilde{f})(x)g(x) \\
&= \phi_{Y_{m-n}}(\alpha_{0,m-n}(\widetilde{f}))(g)(x).
\end{align*}
Hence we have
\begin{equation} \label{eq: left action of f tilde as compacts}
\phi_{Y_{m-n}}(\alpha_{0,m-n}(\widetilde{f}))=\sum_{j=1}^p \Theta_{\alpha_{m-n}(\widetilde{f}),\alpha_{m-n}(\eta_j)}.
\end{equation}

We use \autoref{eq: action decomp for alpha} to get
\begin{align*}
& i_{Y,n}(\alpha_{n,m}(f)) \\
&\quad= i_{Y,n}\Big(\sum_{i=1}^l\alpha_n(\xi_i)\cdot\alpha_{0,m-n}(\widetilde{f})\Big) \\
&\quad= \sum_{i=1}^li_{Y,n}(\alpha_n(\xi_i))i_{Y,0}(\alpha_{0,m-n}(\widetilde{f})) \\
&\quad=\sum_{i=1}^l\Big(i_{Y,n}(\alpha_n(\xi_i))i_{Y,0}(\alpha_{0,m-n}(\widetilde{f})) + i_{Y,n}(\alpha_n(\xi_i))\Big(\sum_{j=1}^pi_{Y,m-n}(\alpha_{m-n}(\widetilde{f}))i_{Y,m-n}(\alpha_{m-n}(\eta_j))^*\Big) \\ 
&\hspace{8em} - 
i_{Y,n}(\alpha_n(\xi_i))\Big(\sum_{j=1}^pi_{Y,m-n}(\alpha_{m-n}(\widetilde{f}))i_{Y,m-n}(\alpha_{m-n}(\eta_j))^*\Big)\Big) \\
&\quad= \sum_{i=1}^l\Big(i_{Y,n}(\alpha_n(\xi_i))\Big(\sum_{j=1}^pi_{Y,m-n}(\alpha_{m-n}(\widetilde{f}))i_{Y,m-n}(\alpha_{m-n}(\widetilde{f}))^*\Big) \\
&\hspace{8em} - i_{Y,n}(\alpha_n(\xi_i))\Big(\sum_{j=1}^p(i_{Y,m-n}(\alpha_{m-n}(\widetilde{f}))i_{Y,m-n}(\alpha_{m-n}(\eta_j))^*)-i_{Y,0}(\alpha_{0,m-n}(\widetilde{f}))\Big)\Big) \\
&\quad= \sum_{i=1}^l\Big(i_{Y,n}(\alpha_n(\xi_i))\Big(\sum_{j=1}^pi_{Y,m-n}(\alpha_{m-n}(\widetilde{f}))i_{Y,m-n}(\alpha_{m-n}(\eta_j))^*\Big) \\
&\hspace{8em} - i_{Y,n}(\alpha_n(\xi_i))\Big(i_Y^{(m-n)}\Big(\sum_{j=1}^p\Theta_{\alpha_{m-n}(\widetilde{f}),\alpha_{m-n}(\eta_j)}\Big)-i_{Y,0}(\alpha_{0,m-n}(\widetilde{f}))\Big)\Big).
\end{align*}
We now use \autoref{eq: left action of f tilde as compacts} to write the above expression for $ i_{Y,n}(\alpha_{n,m}(f))$ as
\begin{align*}
 i_{Y,n}(\alpha_{n,m}(f)) = \sum_{i=1}^l \Big(i_{Y,n}&(\alpha_n(\xi_i))\Big(\sum_{j=1}^pi_{Y,m-n}(\alpha_{m-n}(\widetilde{f}))i_{Y,m-n}(\alpha_{m-n}(\eta_j))^*\Big) \numberthis \label{eq: new exp} \\ 
&- i_{Y,n}(\alpha_n(\xi_i))\Big(i_Y^{(m-n)}\big(\phi_{Y_{m-n}}(\alpha_{0,m-n}(\widetilde{f}))\big)-i_{Y,0}(\alpha_{0,m-n}(\widetilde{f}))\Big)\Big).
\end{align*}
Since 
\[
q_Y\big(i_Y^{(m-n)}\big(\phi_{Y_{m-n}}(\alpha_{0,m-n}(\widetilde{f}))\big)-i_{Y,0}(\alpha_{0,m-n}(\widetilde{f}))\big) =0,
\]
we can apply the quotient map $q_Y$ to \autoref{eq: new exp} to see that
\begin{align*}
j_{Y,n}(\alpha_{n,m}(f)) &= \sum_{i=1}^l \Big(j_{Y,n}(\alpha_n(\xi_i))\Big(\sum_{j=1}^p j_{Y,m-n}(\alpha_{m-n}(\widetilde{f}))j_{Y,m-n}(\alpha_{m-n}(\eta_j))^*\Big)\Big) \\
&= \sum_{i=1}^l \Big(\zeta^\OO(j_{X,n}(\xi_i))\Big(\sum_{j=1}^p \zeta^\OO(j_{X,m-n}(\widetilde{f}))\zeta^\OO(j_{X,m-n}(\eta_j))^*\Big)\Big) \\
&\in \range(\zeta^\OO),
\end{align*}
which proves that \eqref{eq: j_Y(f) in range zeta^OO} holds.

Recall that $g \in C_c(\Lambda^\infty)$, and that we are aiming to show that $j_{Y,n}(g) \in \range(\zeta^\OO)$. Fix $\varepsilon>0$. We will find $a \in C^*(\Lambda,c)$ such that $\lv j_{Y,n}(g) - \zeta^\OO(a) \rv < \varepsilon$. By \autoref{lem: nice nbhds of infinite paths}, we can choose finitely many open subsets $U_1, \dotsc, U_l$ of $\Lambda^n$ such that for each $i \in \{1,\dotsc,l\}$, $\overline{U_i}$ is a compact $s$-section, and $\{ Z(U_i) : 1 \le i \le l \}$ is an open cover of $\supp(g)$. As in \cite[Remark~2.9]{LPS2014}, choose a partition of unity $g_1,\dotsc,g_l$ subordinate to $\{ Z(U_i) \cap \supp(g) : 1 \le i \le l \}$. For each $i \in \{1,\dotsc,l\}$, define $h_i\colon \Lambda^\infty \to \C$ by
\[
h_i(x) \coloneqq \begin{cases}
g(x)g_i(x)\ &\text{if } x \in \supp(g), \\
0 \ &\text{if } x \notin \supp(g).
\end{cases}
\]
Then $g = \sum_{i=1}^l h_i$. Since we have $h_i \in C_c(Z(U_i))$ for each $i \in \{1,\dotsc,l\}$, we can apply \autoref{lem: C_0(Z(U)) SWT} to approximate $h_i$ by a sum of $p_i$ elements of $\AA_{U_i}$; that is, for each $i \in \{1,\dotsc,l\}$ and each $j \in \{1,\dotsc,p_i\}$, there exist $m_{i,j} \in \N^k$ with $m_{i,j} \ge n$, a precompact open $s$-section $V_{i,j} \subseteq \Lambda^{m_{i,j}-n}$ satisfying $r(\overline{V_{i,j}}) \subseteq s(U_i)$, and $f_{i,j} \in C_c(\Lambda^{m_{i,j}})$ with $\osupp(f_{i,j}) \subseteq U_i V_{i,j}$, such that
\begin{equation} \label{eq: uniform norm < epsilon/l}
\Big\lv h_i - \sum_{j=1}^{p_i} f_{i,j} \circ \rho_{m_{i,j},\infty} \Big\rv_\infty < \frac{\varepsilon}{l}.
\end{equation}
Since $h_i - \sum_{j=1}^{p_i} f_{i,j} \circ \rho_{m_{i,j},\infty} \in C_c(Z(U_i))$ for each $i \in \{1,\dotsc,l\}$, \autoref{lem: uniform norm is Y_n norm} and \autoref{eq: uniform norm < epsilon/l} together imply that
\begin{equation} \label{eq: Y_n norm < epsilon/l}
\Big\lv h_i - \sum_{j=1}^{p_i} \alpha_{n,m_{i,j}}(f_{i,j}) \Big\rv_{Y_n} =\, \Big\lv h_i - \sum_{j=1}^{p_i} f_{i,j} \circ \rho_{m_{i,j},\infty} \Big\rv_{Y_n} =\, \Big\lv h_i - \sum_{j=1}^{p_i} f_{i,j} \circ \rho_{m_{i,j},\infty} \Big\rv_\infty < \frac{\varepsilon}{l}.
\end{equation}
Since we have established that \eqref{eq: j_Y(f) in range zeta^OO} is true, we know that for each $i \in \{1,\dotsc,l\}$, we have 
\[
\sum_{j=1}^{p_i} j_{Y,n}(\alpha_{n,m_{i,j}}(f_{i,j})) \in \range(\zeta^\OO),
\]
and hence we can choose $a_i \in C^*(\Lambda,c)$ with $\zeta^\OO(a_i) = \sum_{j=1}^{p_i} j_{Y,n}(\alpha_{n,m_{i,j}}(f_{i,j}))$. Let $a \coloneqq \sum_{i=1}^l a_i \in C^*(\Lambda,c)$. Then we have $\zeta^\OO(a) = \sum_{i=1}^l j_{Y,n}\big(\sum_{j=1}^{p_i} \alpha_{n,m_{i,j}}(f_{i,j})\big)$, and hence \autoref{eq: Y_n norm < epsilon/l} implies that
\begin{align*}
\lv j_{Y,n}(g) - \zeta^\OO(a) \rv &= \Big\lv j_{Y,n}\Big(\sum_{i=1}^l h_i\Big) - \sum_{i=1}^l j_{Y,n}\Big(\sum_{j=1}^{p_i} \alpha_{n,m_{i,j}}(f_{i,j})\Big) \Big\rv \\
&= \Big\lv \sum_{i=1}^l j_{Y,n}\Big(h_i - \sum_{j=1}^{p_i} \alpha_{n,m_{i,j}}(f_{i,j})\Big) \Big\rv \\
&\le \sum_{i=1}^l \Big\lv h_i - \sum_{j=1}^{p_i} \alpha_{n,m_{i,j}}(f_{i,j}) \Big\rv_{Y_n} \\
&< \varepsilon.
\end{align*}
Therefore, $\zeta^\OO$ is surjective.

To see that $\zeta^\OO$ is injective, we will apply \cite[Corollary~4.14]{CLSV2011}. Part~(i) of \cite[Corollary~4.14]{CLSV2011} holds because the gauge action $\gamma$ on $\OO(Y)$ satisfies $\gamma_z(\zeta^\OO(j_X(f)))=z^n\zeta^\OO(j_X(f))$ for all $n\in\N^k$, $f\in X_n$, and $z\in\T^k$. For part~(ii) of \cite[Corollary~4.14]{CLSV2011}, we first observe that $j_Y$ is isometric by \cite[Theorem~4.1]{SY2010}. Then for each $f,g \in C_0(\Lambda^0)$, we have 
\[
\zeta^\OO(j_{X,0}(f)) = \zeta^\OO(j_{X,0}(g)) \implies j_{Y,0}(\alpha_0(f)) = j_{Y,0}(\alpha_0(g)) \implies \alpha_0(f) = \alpha_0(g),
\]
and so \autoref{lem: props of alpha}~\autoref{item: alpha_n is injective} implies that $f = g$. Hence $\zeta^\OO|_{j_X(C_0(\Lambda^0))}$ is injective, and \cite[Corollary~4.14~(ii)]{CLSV2011} holds. Therefore, \cite[Corollary~4.14]{CLSV2011} implies that $\zeta^\OO$ is injective.
\end{proof}

\begin{remark}
In \cite[Section~6]{KPS2015TAMS}, Kumjian, Pask, and Sims use a $\T$-valued $2$-cocycle $c$ on a row-finite $k$-graph $\Lambda$ with no sources to construct a continuous $\T$-valued $2$-cocycle $\sigma_c$ on the path groupoid $\GG_\Lambda$. They then prove that the twisted groupoid $C^*$-algebra $C^*(\GG_\Lambda,\sigma_c)$ (built using Renault's construction in \cite{Renault1980}) is isomorphic to the twisted $k$-graph $C^*$-algebra $C^*(\Lambda,c)$ (denoted by $C_{\KPS}^*(\Lambda,c)$ in \autoref{item: generalising C*(Lambda,c)}). It is not clear that the construction of $\sigma_c$ in \cite{KPS2015TAMS} will give rise to a continuous $\T$-valued $2$-cocycle $\sigma_c$ on Yeend's boundary-path groupoid $\GG_\Lambda$, for $\Lambda$ a proper, source-free topological $k$-graph. It would be interesting to see if one could modify the construction in \cite{KPS2015TAMS} to cater for topological $k$-graphs, and then prove that there is an isomorphism between the twisted Cuntz--Krieger algebra $C^*(\Lambda,c)$ defined in \autoref{def: the C*s}, and the twisted groupoid $C^*$-algebra $C^*(\GG_\Lambda,\sigma_c)$. 
\end{remark}

\vspace{0.5em}
\bibliographystyle{amsplain}
\makeatletter\renewcommand\@biblabel[1]{[#1]}\makeatother
\bibliography{AB1_references}

\end{document}

%% file: preamble.tex
\usepackage[left=2cm, right=2cm, top=2.5cm, bottom=2cm,headsep=1cm,footskip=0.5cm]{geometry}
\usepackage[hang, flushmargin]{footmisc}
\usepackage{amsmath, amsthm, amssymb}
\usepackage{mathtools}
\usepackage{tabularx}
\usepackage{enumitem}
\usepackage{tikz, tikzsymbols, tikz-cd}
\usepackage{hyperref, aliascnt}

\allowdisplaybreaks

\setenumerate{listparindent=\parindent}

\setlist[enumerate]{font=\normalfont}

\numberwithin{equation}{section}

\def\equationautorefname~#1\null{Equation~(#1)\null}
\def\itemautorefname~#1\null{#1\null}
\def\sectionautorefname~#1\null{Section~#1\null}
\def\subsectionautorefname~#1\null{Section~#1\null}

\newcommand\numberthis{\addtocounter{equation}{1}\tag{\theequation}}

\newtheorem{theorem}{Theorem}[section]

\newaliascnt{thm}{theorem}
\newtheorem{thm}[thm]{Theorem}
\aliascntresetthe{thm}

\newaliascnt{lemma}{theorem}
\newtheorem{lemma}[lemma]{Lemma}
\aliascntresetthe{lemma}

\newaliascnt{prop}{theorem}
\newtheorem{prop}[prop]{Proposition}
\aliascntresetthe{prop}

\newaliascnt{cor}{theorem}
\newtheorem{cor}[cor]{Corollary}
\aliascntresetthe{cor}

\newaliascnt{claim}{theorem}

\aliascntresetthe{claim}

\theoremstyle{definition}

\newaliascnt{definition}{theorem}
\newtheorem{definition}[definition]{Definition}
\aliascntresetthe{definition}

\newaliascnt{notation}{theorem}
\newtheorem{notation}[notation]{Notation}
\aliascntresetthe{notation}

\theoremstyle{remark}

\newaliascnt{remark}{theorem}
\newtheorem{remark}[remark]{Remark}
\aliascntresetthe{remark}

\newaliascnt{remarks}{theorem}
\newtheorem{remarks}[remarks]{Remarks}
\aliascntresetthe{remarks}

\newaliascnt{example}{theorem}
\newtheorem{example}[example]{Example}
\aliascntresetthe{example}

\newaliascnt{examples}{theorem}

\aliascntresetthe{examples}

\newaliascnt{question}{theorem}

\aliascntresetthe{question}

\def\C{\mathbb{C}}

\def\N{\mathbb{N}}

\def\T{\mathbb{T}}
\def\Z{\mathbb{Z}}
\def\AA{\mathcal{A}}
\def\BB{\mathcal{B}}

\def\FF{\mathcal{F}}
\def\GG{\mathcal{G}}
\def\HH{\mathcal{H}}

\def\KK{\mathcal{K}}
\def\LL{\mathcal{L}}

\def\OO{\mathcal{O}}

\def\TT{\mathcal{T}}

\def\VV{\mathcal{V}}
\def\WW{\mathcal{W}}

\def\NO{\mathcal{NO}}
\def\NT{\mathcal{NT}}
\def\vecspan{\operatorname{span}}
\def\clspan{\overline{\operatorname{span}}}
\def\supp{\operatorname{supp}}
\def\osupp{\operatorname{osupp}}
\def\range{\operatorname{range}}

\def\id{\operatorname{id}}
\def\Ind{\operatorname{Ind}}
\def\sce{\operatorname{sce}}
\def\fin{\operatorname{fin}}
\def\cov{\operatorname{cov}}
\def\Mor{\operatorname{Mor}}
\def\Obj{\operatorname{Obj}}
\def\KPS{\operatorname{KPS}}
\newcommand{\la}{\ensuremath{\langle}}
\newcommand{\ra}{\ensuremath{\rangle}}
\newcommand{\lav}{\ensuremath{\lvert}}
\newcommand{\rav}{\ensuremath{\rvert}}
\newcommand{\lv}{\ensuremath{\lVert}}
\newcommand{\rv}{\ensuremath{\rVert}}